\documentclass[12pt]{amsart}

\usepackage{amsfonts,amsmath,amssymb,amsthm}
\usepackage{hyperref}
\hypersetup{breaklinks=true}
\usepackage{rotating}
\usepackage{array}

\numberwithin{equation}{section}
\theoremstyle{plain}
\newtheorem{theorem}[equation]{Theorem} 
\newtheorem{lemma}[equation]{Lemma}    
\newtheorem{proposition}[equation]{Proposition}   
\newtheorem{corollary}[equation]{Corollary}

\newtheorem{question}[equation]{Question}

\newcolumntype{M}[1]{>{\raggedright}m{#1}}

\newcounter{thm}

\newtheorem{main_theorem}[thm]{Theorem}

\theoremstyle{remark}
\newtheorem*{remark*}{Remarks} 
\newtheorem*{example*}{Example}

\theoremstyle{definition}
\newtheorem{definition}[equation]{Definition}
\newtheorem{remark}[equation]{Remark}
\newtheorem{example}[equation]{Example}

\newcommand{\im}{\operatorname{Im}}

\newcommand{\Ker}{\operatorname{Ker}}

\newcommand{\restr}{\mbox{\Large \(|\)\normalsize}}

\begin{document}

\title[]
{Non-vanishing for group $L^p$-cohomology of solvable and semisimple Lie groups}
 
\author{Marc Bourdon and Bertrand R\'emy} 
\maketitle

\begin{abstract}
We obtain non-vanishing of group $L^p$-cohomology of Lie groups for $p$ large and when the degree is equal to the rank of the group. 
This applies both to semisimple and to some suitable solvable groups. 
In particular, it confirms that Gromov's question on vanishing below the rank is formulated optimally. 
To achieve this, some complementary vanishings are combined with the use of spectral sequences.
To deduce the semisimple case from the solvable one, we also need comparison results between various theories for $L^p$-cohomology, allowing the use of quasi-isometry invariance.

\vspace*{2mm} \noindent{2010 Mathematics Subject Classification: } 20J05, 20J06, 22E15, 22E41, 53C35, 55B35, 57T10, 57T15.
%
%
%
%

\vspace*{2mm}

\noindent{Keywords and phrases: } $L^p$-cohomology, Lie~group, symmetric space, quasi-isometric invariance, spectral sequence, cohomology (non-)vanishing, root system.
\end{abstract}

\setlength{\parskip}{\smallskipamount}

\tableofcontents

\setlength{\parskip}{\medskipamount}

\newpage

\section*{Introduction}
\label{1}

This paper deals with several variants of group $L^p$-cohomology, where $p$ is 
a real number in $(1, +\infty)$. 

The first one, the \emph{continuous $L^p$-cohomology}~of locally compact second countable groups $G$, is the continuous cohomology \cite[Chap.\! IX]{BW} with coefficients in the right-regular representation on $L^p(G)$; we denote it by ${\rm H}^*_{\mathrm {ct}} ( G, L^p(G))$.
We also consider the associated \emph{reduced continuous $L^p$-cohomology}, denoted by $\overline{{\rm H}^*_{\mathrm {ct}}} ( G, L^p(G))$ 
(it is the largest Hausdorff quotient of the previous one).
It is known that both topolical vector spaces ${\rm H}^*_{\mathrm {ct}} ( G, L^p(G))$ and $\overline{{\rm H}^*_{\mathrm {ct}}} ( G, L^p(G))$ are invariant for the equivalence relation given by quasi-isometries between groups $G$ as before when equipped with a left-invariant proper metric (see \cite[Theorem 1.1]{BR} and \cite{SaSc}).

One variant of $L^p$-cohomology makes sense for $C^\infty$ manifolds: it is the \emph{de Rham $L^p$-cohomology}, denoted by $L^p {\rm H}^*_{\rm dR}(M)$ for a $C^\infty$ manifold $M$. 
It was studied thoroughly by P.~Pansu and, as the name suggests, it is defined by imposing $L^p$-integrability conditions on differential forms on the manifold (and on their differentials). 
There are also intermediate variants, for instance the \emph{asymptotic $L^p$-cohomology} of suitable metric spaces, which elaborates on simplicial cohomology by adding $L^p$-integrability conditions; it is denoted by $L^p {\rm H}^*_{\rm AS}(X)$ for a suitable measured metric space $(X,d,\mu)$. 
Reduced quotients are also considered in these contexts. 

When it makes sense, comparison results between these cohomologies are often available (see for instance \cite{BR}, \cite{SaSc} and Appendix \ref{app - asymptotic and de Rham} in this paper; again, this part owes a lot to P.~Pansu's work). 

Beyond these comparisons, our main goal remains to exhibit some sufficient conditions for vanishing and non-vanishing of $L^p$-cohomologies for topological groups, 
taking into account the degree of the cohomology space and the exponent $p$. 

The topological groups we are considering in the present paper for these (non-)vanishing questions are connected Lie groups. 
This enables us to use many differential geometric and combinatorial tools.
The most popular Lie groups are the semisimple ones, and this is the family about which we present the first results in this introduction, but we will see that we are quickly led to considering solvable non-unimodular groups. 
This is explained by the fact that some contraction arguments are crucial for our purposes, and this is made possible by the aformentioned quasi-isometric invariance of group $L^p$-cohomology combined with Iwasawa decompositions.

\subsection{Semisimple groups and quasi-isometry invariance} 
\label{ss - ss and QI}
Let us start with semisimple real Lie groups. 
Our work is motivated by the following question, asked by M.~Gromov \cite[p.\! 253]{G}.

\begin{question}
Let $G$ be a semisimple real Lie group.
We assume that $l=\mathrm{rk}_{\bf R} (G) \geqslant 2$. 
Let $k$ be an integer $<l$ and $p$ be a real number $>1$. 
Do we have: ${\rm H}^k_{\mathrm {ct}} \bigl(G, L^p(G) \bigr) = \{0\}$?
\end{question}

In degree 1, a general result of Pansu \cite{P2} and Cornulier-Tessera \cite{CT} shows that ${\rm H}^1_{\mathrm {ct}} ( G, L^p(G) ) = \{0\}$ for every $p>1$ and
every connected Lie group $G$, unless $G$ is Gromov hyperbolic or amenable unimodular. 
In a previous paper, we obtained some partial vanishing results which lead to the latter result when the Lie groups are semisimple of rank $\geqslant 2$, and admissible in the sense that the solvable radical of some (maximal) parabolic subgroup is quasi-isometric to a real hyperbolic space \cite[Corollary 1.6]{BR}. 
This result in degree 1 can also be proved via the fixed point property for continuous affine isometric actions of higher rank semisimple groups on $L^p$-spaces \cite{BFGM}.
In general, we show in \cite{BR} the existence, for any admissible semisimple Lie group $G$ and any $p >1$, of an interval of degrees out of which the spaces ${\rm H}^k_{\mathrm {ct}} (G, L^p(G))$ vanish.

Our main result in the present paper is complementary to Gromov's question; it is the following. 

\begin{main_theorem}
\label{introduction-theorem1} 
Let $G$ be a semisimple real Lie group with finite center and let $l=\mathrm{rk}_{\bf R} (G)$. 
\smallskip 
\begin{itemize}
\item[{\rm (i)}]~We have: $\overline{{\rm H}^l_{\mathrm {ct}}} \bigl(G, L^p(G) \bigr) \neq \{0\}$ for any large enough $p>1$. 
\item[{\rm (ii)}]~For every $k >l$, we have: ${\rm H}^k_{\mathrm {ct}} \bigl(G, L^p(G) \bigr) = \{0\}$ for any large enough $p>1$. 
\end{itemize}
\end{main_theorem}

This result is proved thanks to the following line of arguments. 
We introduce an Iwasawa decomposition $G = KAN$. 
Geometrically, if $X$ denotes the symmetric space of $G$ and if $F$ denotes the maximal flat attached to the maximal ${\bf R}$-split torus $A$ in $G$, then the subgroup $K$ can be chosen to be the stabilizer of a point in $F$; the subgroup $N$ consists of the unipotent elements in a parabolic subgroup defined by a regular element in the boundary $\partial_\infty F$ of $F$. 
Since $G$ has finite center, the group $K$ is compact (it is in fact a maximal compact subgroup in $G$). 
We have the following identifications: 
\[
{\rm H}^*_{\mathrm {ct}} \bigl(G, L^p(G) \bigr) \simeq {\rm H}^*_{\mathrm {ct}} \bigl(AN, L^p(AN) \bigr) \simeq L^p {\rm H}^*_{\rm AS}(AN) \simeq L^p {\rm H}^*_{\rm dR}(AN),
\]
and similar ones for reduced cohomology. 
The first identification comes from quasi-isometric invariance \cite[Theorem 1.1]{BR}, the second one is a comparison between continuous $L^p$-cohomology and asymptotic $L^p$-cohomology proved in \cite[Theorem 3.6]{BR} and the last one is given here by Theorem \ref{appendix-theorem} (proved in Appendix \ref{app - asymptotic and de Rham}). 
This reduction explains why the main part of the paper focuses on solvable Lie groups. 
The latter situation is investigated in the remaining two subsections of this introduction and it leads to the desired vanishing and non-vanishing results above. 
See Theorem \ref{introduction-theorem4} below for the conclusion of the argument. 


Since the proof of our main result on semisimple groups is spread all over our paper, here is a summary of the strategy. 

\smallskip 

{\it Step 1: quasi-isometry invariance and Iwasawa decomposition}.--- The decomposition $G=KAN$ implies that $G$ is quasi-isometric to the solvable group $A \ltimes N$. 
Invariance of $L^p$-cohomology under quasi-isometric invariance then gives 
${\rm H}^*_{\rm ct}\bigl( G, L^p(G) \bigr) = {\rm H}^*_{\rm ct}\bigl( AN, L^p(AN) \bigr)$ and 
$\overline{{\rm H}^*_{\rm ct}} \bigl( G, L^p(G) \bigr) = \overline{{\rm H}^*_{\rm ct}}\bigl( AN, L^p(AN) \bigr)$.
We denote $R = AN = A \ltimes N$ so that we have $R \simeq \mathbf{R}^D$ where $D$ is the dimension of the Riemannian symmetric space $G/K$, as well as $A \simeq \mathbf{R}^l$ where $l$ is the real rank of $G$. 

\smallskip 

{\it Step 2: Poincar\'e duality and vanishing (after Pansu)}.--- 
Poincar\'e duality reduces the proof of Theorem A to showing that for $p$ close enough to 1, we have $L^p{\rm H}^k_{\rm dR}(R) = \{ 0 \}$ for every $k < D-l$ and that $L^p{\rm H}^{D-l}_{\rm dR}(R)$ is Hausdorff and non-zero. 
Arguments due to Pansu show that a certain contraction condition on the $A$-action on $N$ (called (nC) in Theorem \ref{introduction-theorem} below) implies the assertions about vanishing and Hausdorff property (see Corollary \ref{lie-corollary1}). 

\smallskip 
We concentrate now on the non-vanishing $L^p{\rm H}^{D-l}_{\rm dR}(R) \neq \{ 0 \}$. 

\smallskip 

{\it Step 3: Actions of abelian groups on Heintze groups and spectral sequences}.--- 
We pass now from the decomposition $R = A \ltimes N$ to the decomposition $R = B \ltimes ( \{ e^{t\xi} \}_{t \in \mathbf{R}} \ltimes N)$ where $\xi$ is given by condition (nC) as before, $H = \{ e^{t\xi} \}_{t \in \mathbf{R}} \ltimes N$ is a Heintze group and $B$ is a suitable (abelian) complement  of $\{ e^{t\xi} \}_{t \in \mathbf{R}}$ in $A$. 
The motivation for the decomposition $R = B \ltimes H$ is the possibility to use a spectral sequences and the already proved vanishings to obtain the identification: 
$L^p{\rm H}^{D-l}_{\rm dR}(R) \simeq L^p \bigl( B , L^p{\rm H}^{D-l}_{\rm dR}(H) \bigr)^B$. 
We conclude by exhibiting $L^p$ de Rham cohomology classes $c$ on $H$ such that the map $b \mapsto \Vert b \cdot c \Vert_p$ belongs to $L^p(B)$ for the $B$-action on $c$ given by (pull-back of) conjugation by $B$ on $H$. 
This is the content of the (technical) Proposition \ref{non-vanishing-proposition}, which uses condition (nT), a condition of non-triviality required for the action of all elements in the Lie algebra of $A$ (see Question \ref{introduction-question} below). 


Before we definitely move to the framework of solvable groups, let us mention the case when $p=2$ for semisimple groups. 
This situation is more directly relevant to representation theory and was considered by A.~Borel decades ago \cite{Borel}. 
The main results are: 
\begin{itemize}
\item $\overline{{\rm H}^k_{\mathrm {ct}}} \bigl( G, L^2(G) \bigr) = 0$ unless $k = \frac{D}{2}$,
\item  ${\rm H}^k_{\mathrm {ct}} \bigl( G, L^2(G) \bigr) \neq 0$ at least for $k \in \bigl(\frac{D}{2} - \frac{l_0}{2}, \frac{D}{2} +\frac{l_0}{2}\bigr]$,
\end{itemize}
where $D$ is the dimension of the Riemannian symmetric space $G/K$ and $l_0$ is the difference between the complex rank of $G$ and the complex rank of $K$. 
In the case $G = {\rm SL}_n({\bf R})$, one has $D = \frac{n^2 + n -2}{2}$ and 
$l_0 = \lfloor \frac{n-1}{2} \rfloor$. 
See also \cite{BFS} for related results about vanishing of the reduced $L^2$-cohomology. 
Since the dimension $D$ is a quadratic polynomial in the rank of the group $G$, A.~Borel's results show that the assumption that $p$ should be large enough in our results is necessary.

\subsection{Solvable groups and contractions} 
\label{ss - solvable and contraction}
As explained above, we are henceforth dealing with solvable Lie groups until the end of the introduction. 
The following result on cohomology vanishing is a consequence of \cite[Corollaire 53]{Pa99}:

\begin{theorem}
\label{introduction-theorem} 
Let $R$ be a connected Lie group of the form $A \ltimes N$ with $A \simeq {\bf R}^l$ and $l \geqslant 1$. 
Let $\frak a$ and $\frak n$ be the Lie algebras of $A$ and $N$, respectively.
Suppose that the group $R$ satisfies the following contraction property: 
\begin{itemize}
\item[{\rm (nC)}] there exists an element $\xi \in \frak a$ such that all the eigenvalues of $\mathrm{ad} \xi \restr _{\frak n}$
have negative real parts. 
\end{itemize}
Then for $p > 1$ large enough and for all $k >l$, we have the vanishings: 

\smallskip 

\centerline{$L^p\mathrm {H}_{\mathrm{dR}}^k(R) = \{0\}$.}
\end{theorem} 

Note that the assumptions force $N$ to be nilpotent and contractible, and $R$ to be solvable, non-unimodular and diffeomorphic to ${\bf R}^D$, where
$D= \dim(R)$ (the notation $D$ is consistent with the previous one on dimensions of symmetric spaces by Iwasawa decomposition). 
We also mention the fact that condition (nC) already appears in N.~Varopoulos' paper \cite[OV.3 p.\! 799]{Var}. 

\smallskip 
We address here the following question:

\begin{question}
\label{introduction-question}
Let $R$ be a solvable connected Lie group as in the previous theorem, satisfying in particular condition ${\rm (nC)}$. 
Suppose additionally that $R$ satisfies the following non-triviality condition: 
\begin{itemize}
\item[{\rm (nT)}]for every non-trivial $X \in \frak a$, the operator $\mathrm{ad} X \restr _{\frak n}$ admits 
an eigenvalue with non-zero real part. 
\end{itemize}
Do we have $L^p\overline{\mathrm{H}_{\mathrm{dR}}^l}(R) \neq \{0\}$ for $p >1$ large enough?
\end{question}

If we denote by ${\rm sp}(u)$ the spectrum (in the field of complex numbers) of an endomorphism $u$ of a finite-dimensional real vector space, the above conditions can be reformulated as follows: 
\begin{itemize}
\item[{\rm (nC)}]there exists $\xi \in \frak a$ such that ${\rm sp}(\mathrm{ad} \xi \restr _{\frak n}) \subset {\bf R}_-^\times \oplus i{\bf R}$, 
\item[{\rm (nT)}]for every $X \in \frak a \setminus \{ 0 \}$, we have: ${\rm sp}(\mathrm{ad} X \restr _{\frak n}) \cap ({\bf R}^\times \oplus i{\bf R}) \neq \varnothing$.
\end{itemize} 

Condition (nC)+(nT) forces $N$ to be the nilpotent radical of $R$.
When $l=1$, condition (nC) implies trivially condition (nT); moreover the groups $R$ that satisfy (nC) form precisely the class of connected Lie groups that carry a left-invariant Riemannian 
metric of negative curvature \cite{Hze}. 
For them, it is known that Question \ref{introduction-question} admits a positive answer \cite{P2, CT}. 
The goal of the paper is to enlarge the family of groups for which Question \ref{introduction-question} is known to have a positive answer.
In Section \ref{s - proof of non-vanishings}, we prove two non-vanishing results described below.
Note that at least since we are using spectral sequences techniques, vanishing results are also useful in the proof of non-vanishing ones; this explains why condition (nC) is made before condition (nT). 

The first positive result makes a commutativity assumption on the nilpotent radical of the solvable group. 

\begin{main_theorem}
\label{introduction-theorem2} 
The answer to Question \ref{introduction-question} is yes, if one assumes in addition that 
$N \simeq {\bf R}^n$ with $n \geqslant 1$.
\end{main_theorem}

The second positive result makes a rank assumption on the solvable group, \emph{i.e.}~ a dimension assumption on the quotient of the group by its nilpotent radical. 

\begin{main_theorem}
\label{introduction-theorem3} 
The answer to Question \ref{introduction-question} is yes, if one assumes in addition that 
$l=2$, \emph{i.e.}~ if $A \simeq {\bf R}^2$.
\end{main_theorem}

Fundamental examples of groups satisfying conditions (nC) and (nT) are provided by the groups $AN$ that appear in the Iwasawa
decompositions $KAN$ of the semisimple Lie groups with finite center. 

\begin{main_theorem}
\label{introduction-theorem4} 
The answer to Question \ref{introduction-question} is yes for the groups $AN$ that appear in the Iwasawa decompositions $KAN$ of the semisimple Lie groups with finite center.
\end{main_theorem}

Together with the reduction contained in Subsection \ref{ss - ss and QI}, the latter result proves Theorem \ref{introduction-theorem1} on semisimple groups.

\subsection{Sketch of proof of the non-vanishing theorems}
\label{introduction-strategy}
Recall that $D = \dim(R)$. 
By Poincar\'e duality (see Proposition \ref{poincare-proposition}), proving vanishing and non-vanishing as stated in 
Theorem \ref{introduction-theorem} and Question \ref{introduction-question}, is equivalent to showing
that for $p > 1$ close enough to $1$:
\begin{enumerate}
\item $L^p\mathrm {H}_{\mathrm{dR}}^k(R) = \{0\}$ for every $k <D-l$,
\item $L^p\mathrm {H}_{\mathrm{dR}}^{D-l}(R)$ is Hausdorff and non-zero.
\end{enumerate}
The assumption on the existence of $\xi \in \frak a$ satisfying condition (nC), in combination with Pansu's results on $L^p$-cohomology \cite{Pa99, P1, Pa09}, 
imply that item (1) holds, as well as the Hausdorff property in item (2) 
(see Corollary \ref{lie-corollary1} for a proof).

It remains to establish that $L^p\mathrm {H}_{\mathrm{dR}}^{D-l}(R) \neq \{0\}$. For that, we use $\xi$ as in condition (nC) in order to decompose $R$ as follows. 
Write $\frak a = {\bf R} \xi \oplus \frak b$, 
where the second factor is $\frak b:=\{X \in \frak a: \mathrm{trace}(\mathrm{ad}X) = 0\}$, the Lie algebra of a connected
Lie subgroup $B < A$ isomorphic to ${\bf R}^{D-1}$.
Then $R$ can be expressed as $R = B \ltimes H$, with $H = \{e^{t\xi}\}_{t \in {\bf R}} \ltimes N$. 
Again the assumption on $\xi$ and Pansu's results on 
$L^p$-cohomology, give a precise rather simple description of the $L^p$-cohomology of $H$ 
for $p > 1$ close to $1$ -- see Corollaries \ref{lie-corollary1} and \ref{lie-corollary2}. In particular the above items (1) and (2) hold for $H$, 
in other words 
$L^p\mathrm {H}_{\mathrm{dR}}^{k} (H)$ vanishes for $k < \dim(H) -1 = D-l$, and $L^p\mathrm {H}_{\mathrm{dR}}^{D-l} (H)$ is Hausdorff and non-zero.

Now, by using a spectral sequence argument taken from \cite{BR} -- see Corollary \ref{coho-corollary}, we obtain a linear isomorphism:
\begin{eqnarray}\label{introduction-eqn}
L^p\mathrm {H}_{\mathrm{dR}}^{D-l}(R) \simeq L^p \bigl(B, L^p\mathrm {H}_{\mathrm{dR}}^{D-l} (H)\bigr)^B,
\end{eqnarray}
where $B$ acts by translations on itself and by conjugacy on $L^p\mathrm {H}_{\mathrm{dR}}^{D-l} (H)$.
The $B$-invariance implies that the right hand side space is isomorphic to
$$\Bigl\{\psi \in L^p\mathrm {H}_{\mathrm{dR}}^{D-l} (H): \int _{B} \bigl\Vert C_{b}^* (\psi) \bigr\Vert^p db < +\infty \Bigr\},$$
where $C_b$ denotes the conjugation by $b \in B$.
Under the assumptions (nC) and (nT), we give a criterion to ensure that the latter space is non-zero (see Proposition \ref{non-vanishing-proposition}). 
Finally we show that the groups in Theorems \ref{introduction-theorem2}, \ref{introduction-theorem3} and \ref{introduction-theorem4}
satisfy the criterion. 

It is worth mentioning that the above strategy would answer affirmatively Question \ref{introduction-question}
in full generality, if the following question admitted a positive answer:
\begin{question}\label{introduction-question2} Let $N$ be a connected simply connected nilpotent Lie group.
Let $X_1, X_2, \dots , X_k$ be non-trivial left-invariant vector fields on $N$. Does $N$ admit a non-zero, 
compactly supported, $C^1$ function
$f: N \to{\bf R}$, whose integral along every orbit of  $X_i$ $(i = 1, \dots , k)$ is null?
\end{question}
When the fields commute or when $k =2$, we answer affirmatively Question \ref{introduction-question2} --
see Lemmata \ref{criterium-lemma1} and \ref{criterium-lemma2}.

\subsection{Remarks and questions}
{\bf{1)}} We suspect that Condition (nT) in Question \ref{introduction-question} is a necessary condition. 
For example if $\frak a$ contains a non-trivial vector $X_0$ such that $\mathrm{ad} X_0 \restr _{\frak n}$ is \emph{semisimple}
with imaginary eigenvalues, then one has $L^p\overline{\mathrm{H}_{\mathrm{dR}}^k}(R) = \{0\}$
for every $k \geqslant 0$ and $p > 1$.
To see this, observe that for such an $X_0$ the operator $\mathrm{ad} X_0$ acting on ${\rm Lie}(R)$ is skew-symmetric. 
Therefore the left-invariant 
vector field associated to $X_0$ is a Killing vector field on $R$. 
Moreover its flow acts properly on $R$ (since it does on $A$). 
These properties, in combination with Poincar\'e duality, imply vanishing of 
$L^p\overline{\mathrm{H}_{\mathrm{dR}}^k}(R)$ in every degree -- 
see \cite[Proof of Proposition 15]{P1}.

{\bf{2)}} Cornulier told us that conditions (nC) and (nT) should admit geometric characterizations. He claims that a simply connected solvable Lie group
satisfies condition (nC) if and and only if its asymptotic cone (one or every) is bilipschitz homeomorphic to a ${\rm CAT}(0)$-space. Moreover he thinks that condition
(nT) should be equivalent to the non-existence of direct $\mathbf R$-factor in the asymptotic cone.
In the same vein, we notice that the rank $l$ of a condition (nC) group is equal to the dimension of its asymptotic cone, see \cite[Theorem 1.1]{C}.

{\bf{3)}} Nice examples of groups satisfying condition (nC) are provided by Lie groups that carry a left-invariant non-positively curved
Riemannian metric; although these two classes of groups do not exactly coincide (see \cite[Theorem 7.6]{AW}).

{\bf{4)}} Cornulier advertised us of the following potentially interesting quasi-isometric invariant:
$$p(R) := \inf \bigl\{p > 1 ~\vert~ L^p \mathrm H _{\mathrm{dR}} ^l (R) \neq \{0\} \bigr\},$$
for groups $R$ of rank $l$ that satisfy conditions (nC) and (nT). 
(In case the right side set is void we set
$p(R) = +\infty$).  Similarly is defined the invariant $\overline p (R)$ associated to the reduced $L^p$-cohomology.
When $l =1$, these invariants have been considered and computed by Pansu \cite{P2}  (see also \cite{CT} for related results).



\subsection*{Organization of the paper} 
Section \ref{s - dR} is a brief presentation of results on de Rham $L^p$-cohomology of manifolds mainly due to P.~Pansu. 
Section \ref{s - Lie group dR} applies these results to the case of Lie groups, and Section \ref{s - non-vanishing crit} provides a general non-vanishing criterion in terms of spectra of adjoint actions on solvable groups. 
Section \ref{s - proof of non-vanishings} then applies the criterion to prove our first two non-vanishing results, namely Theorems \ref{introduction-theorem2} and \ref{introduction-theorem3} above. 
Section \ref{s - semisimple} is dedicated to semisimple Lie groups: it shows, by combinatorial arguments using Cartan's classification of Riemannian symmetric spaces, that the solvable subgroups arising from Iwasawa decompositions do satisfy the previous criterion; Theorem \ref{introduction-theorem4} follows. 
Section \ref{s - psd} deals with semi-direct products and allows the use of spectral sequences, modulo a comparison result between cohomologies, available when the involved groups are diffeomorphic to ${\bf R}^D$ and which is proved in Appendix \ref{app - asymptotic and de Rham}.

\subsection*{Acknowledgements} 
We thank Pierre Pansu: the present paper elaborates on several of his ideas and results. 
We thank Yves Cornulier for useful remarks and questions.
M.B.\! was partially supported by the Labex Cempi. 

\section{de Rham $L^p$-cohomology}
\label{s - dR}

Pansu's work on de Rham $L^p$-cohomology is dense and subtle. 
In this section, we extract from his papers \cite{P1, Pa09} the ideas and results that are needed in the sequel.
Since \cite{Pa09} is not published yet, and because we only need special cases which require simpler arguments, we include full
proofs of the statements.
We hope that this section could also serve as a gentle introduction to the subject.

\subsection{Definitions}
Let $M$ be a $C^\infty$ Riemannian manifold. We denote by $d\mathrm{vol}$ its Riemannian mesure, and by $\vert v \vert$ the Riemannian length of a vector $v \in TM$.
For $k \in \bf N$, let $\Omega^k(M)$ be the space of $C^\infty$ differential $k$-forms on $M$. 

Let $p\in (1, +\infty)$. The \textit{$L^p$-norm of $\omega \in \Omega^k (M)$}
is $$\Vert \omega \Vert _{L^p \Omega^k} = \bigl(\int _M \vert \omega \vert _m^p ~d\mathrm{vol}(m)\bigr)^{1/p}, $$
$$\mathrm{where~~~~}\vert \omega \vert _m:= \sup \{ \vert \omega (m; v_1, ..., v_k) \vert: v_1, \dots, v_k \in T_m M, ~\vert v_i \vert = 1\}.$$
We denote by $L^p \Omega^k (M)$ the norm completion of the normed space $\{\omega \in \Omega^k (M): \Vert \omega \Vert _{L^p \Omega^k} < +\infty \}$, 
\emph{i.e.}~ the Banach space of $k$-differential forms with measurable $L^p$ coefficients.

Let also 
$$\Vert \omega \Vert _{\Omega^{p,k}}:= \Vert \omega \Vert _{L^p \Omega^k} + \Vert d\omega \Vert _{L^p \Omega^{k+1}}.$$ 
One defines $\Omega^{p,k} (M)$ to be equal to the norm completion of the normed space 
$\{\omega \in \Omega^k (M): \Vert \omega \Vert _{\Omega^{p,k}} < +\infty \}$. 
By construction $\Omega^{p,k} (M)$ is a Banach space and the standard differential operator extends to a bounded operator 
$$d_k: \Omega^{p,k} (M) \to \Omega^{p,k+1} (M),$$
which satisfies $d\circ d =0$.

\begin{definition} The \textit{de Rham $L^p$-cohomology} of $M$ is the cohomology of the complex $\Omega^{p,0} (M) \stackrel{d_0}{\to} 
\Omega^{p,1} (M) \stackrel{d_1}{\to} \Omega^{p,2} (M) \stackrel{d_2}{\to} \dots $. It will be
denoted by $L^p \mathrm{H_{dR}^*} (M)$. Its largest Hausdorff quotient is denoted by 
$L^p \overline{\mathrm{H^*_{dR}}} (M)$ and is called the \textit{reduced de Rham $L^p$-cohomology} of $M$.
\end{definition}

\begin{remark} According to \cite[Theorem 12.8]{GT06}, the inclusion map 
$$\{\omega \in \Omega^* (M): \Vert \omega \Vert _{\Omega^{p,*}} < +\infty \} \subset \Omega^{p, *}(M)$$ induces a 
topological isomorphism in cohomology. Therefore every element in $L^p \mathrm{H_{dR}^*} (M)$ can be represented by a smooth form.
In particular, \textit{when $M$ is a compact manifold}, its de Rham $L^p$-cohomology coincides with the standard de Rham cohomology.
\end{remark}

\subsection{Vanishing} Let $\varphi: M \to M$ be a $C^\infty$ map and $k \in \bf N$.
We denote by $\varphi^*: \Omega^k (M) \to \Omega^k (M)$, $\omega \mapsto \varphi^* (\omega)$,
the associated linear map. In case it induces a bounded operator $\varphi^*: L^p \Omega^k (M) \to L^p \Omega^k (M)$, 
we denote its operator
norm by $\Vert \varphi^* \Vert _{L^p \Omega^k \to L^p \Omega^k}$. 
Otherwise we set $\Vert \varphi^* \Vert _{L^p \Omega^k \to L^p \Omega^k} = +\infty$.
We define similarly $\Vert \varphi^* \Vert _{\Omega^{p,k} \to \Omega^{p,k}}$.

Let $\xi$ be a $C^\infty$ unit complete vector field on $M$, and let $(\varphi _t)_{t \in{\bf R}}$ be its flow.
We assume that $\varphi _t^*: L^p \Omega^k (M) \to L^p \Omega^k (M)$ is bounded for every $t \in \bf R$ and $k \in \bf N$.
For $k \in \bf N^*$, let $\iota _\xi: \Omega^k (M) \to \Omega^{k-1} (M)$ be the inner product with $\xi$.
We observe that $\varphi _t^* \circ \iota _\xi = \iota _\xi \circ \varphi _t^*$. Moreover $\iota_\xi$ contracts the norms 
$\Vert \cdot \Vert _{L^p \Omega^*}$; indeed $\vert \iota_\xi \omega \vert _m \leqslant \vert \omega \vert _m$, since $\vert \xi \vert _m =1$ by assumption. 

\begin{lemma}\label{vanishing-lemma} 
For $t \geqslant 0$ and $k \in \bf N^*$, the linear map: $B_t^k: \Omega^k (M) \to \Omega^{k-1} (M)$ 
defined by
$$B_t^k (\omega) = \int _0^t \varphi _s^* (\iota _\xi \omega) ~ds,$$
induces a bounded linear operator from $\Omega^{p,k}(M)$ to $\Omega^{p, k-1}(M)$, still denoted by $B_t^k$, 
that satisfies the following homotopy relation
$$d \circ B_t^k + B_t^{k+1} \circ d = \varphi _t^* - \mathrm{id},$$
and whose operator norm satisfies 
$$\Vert B_t^k \Vert _{\Omega^{p,k} \to \Omega^{p,k-1}} \leqslant \Vert \varphi _t^* - \mathrm{id}\Vert _{L^p \Omega^k \to L^p \Omega^k} + 
\int _0^t \Vert \varphi _s^* \Vert _{L^p \Omega^k \to L^p \Omega^k} ~ds.$$


\end{lemma}

\begin{proof} For smooth forms the homotopy relation holds as a standard application of the classical Cartan formula 
$$\mathcal L _\xi = d \circ \iota _\xi + \iota _\xi \circ d$$ (see \emph{e.g.}~\cite[Proposition 1.121]{GHL} for this formula). Indeed for $\omega \in \Omega^k(M)$
one has 
$$\varphi _t^* (\omega) - \omega = \int _0^t \frac{d}{ds}(\varphi _s^* \omega)~ds = \int _0^t \varphi _s^* (\mathcal L _\xi \omega) ~ds
= dB_t^k (\omega) + B_t^{k+1}(d\omega).$$

We will prove that $B_t^k$ is a bounded linear operator from $\Omega^{p,k}(M)$ to $\Omega^{p, k-1}(M)$. Since smooth forms are dense in 
$\Omega^{p, *}(M)$
the homotopy formula will remain valid in $\Omega^{p, *}(M)$.

For $\omega \in \Omega^k (M)$, one has by definition
$$\Vert B_t^k (\omega) \Vert _{\Omega^{p, k-1}} = 
\Vert B_t^k (\omega) \Vert _{L^p \Omega^{k-1}} + 
\Vert d B_t^k (\omega) \Vert _{L^p \Omega^k}.$$
From the homotopy relations we get that 
\begin{align*}
\Vert B_t^k (\omega) \Vert _{\Omega^{p, k-1}}&
\leqslant \Vert B_t^k \Vert _{L^p \Omega^k \to L^p \Omega^{k-1}} \cdot \Vert \omega \Vert _{L^p \Omega^k} \\
&+ \Vert \varphi _t^* - \mathrm{id} \Vert _{L^p \Omega^k \to L^p \Omega^k} \cdot \Vert \omega \Vert _{L^p \Omega^k}\\
&+ \Vert B_t^{k+1} \Vert _{L^p \Omega^{k+1} \to L^p \Omega^k} \cdot \Vert d \omega \Vert _{L^p \Omega^{k+1}}.
\end{align*}
Since $\varphi _s^* \circ \iota _\xi = \iota _\xi \circ \varphi _s^*$, 
and since the maps $\iota _\xi: L^p \Omega^* (M) \to L^p \Omega^{*-1} (M)$ are contracting, one has 
$$\Vert B_t^* \Vert _{L^p \Omega^* \to L^p \Omega^{*-1}} \leqslant 
\int _0^t \min \{\Vert \varphi _s^* \Vert _{L^p \Omega^{*-1} \to L^p \Omega^{*-1}} , \Vert \varphi _s^* \Vert _{L^p \Omega^* \to L^p \Omega^*}\}~ds.$$
The expected upper bound for $\Vert B_t^k \Vert _{\Omega^{p,k} \to \Omega^{p,k-1}}$ follows easily.
\end{proof}

We now establish a vanishing result.

\begin{proposition}\label{vanishing-proposition} \cite[Proposition 10]{P1} Let $p \in (1, +\infty)$ and $k \in \bf N^*$. Suppose that there exists $C, \eta >0$, such that for every $t\geqslant 0$:
$$\max _{i = k-1, k} \Vert \varphi _t^* \Vert _{L^p \Omega^i \to L^p \Omega^i} \leqslant C e^{-\eta t}.$$
Then $L^p \mathrm{H}_{\mathrm{dR}}^k (M) = \{0\}$.
\end{proposition}

\begin{proof} First, the assumption (with $i=k$) and the previous lemma imply that the operators 
$B_t^k: \Omega^{p,k}(M) \to \Omega^{p, k-1}(M)$ 
are bounded independently of $t \geqslant 0$. 
Secondly, the assumption (with $i = k-1$) implies that $\Vert \varphi _t^* \Vert _{\Omega^{p,k-1} \to \Omega^{p,k-1}} \to 0$, when $t \to +\infty$.
We claim that these two observations imply that $B_t^k$ converges in norm to an operator 
$B_\infty^k: \Omega^{p,k}(M) \to \Omega^{p, k-1}(M)$, when $t \to +\infty$.
Indeed, by using a change of variable, one has for $0 \leqslant t_1 \leqslant t_2$:
$$B_{t_2}^k - B_{t_1}^k = \int _{t_1}^{t_2} \varphi _s^* \circ \iota _\xi ~ds = \varphi _{t_1}^* \circ B_{t_2 - t_1}^k.$$
Thus the above observations yield:
$$ \Vert B_{t_2}^k - B_{t_1}^k \Vert _{\Omega^{p,k} \to \Omega^{p, k-1}} 
\leqslant \Vert \varphi _{t_1}^* \Vert _{\Omega^{p,k-1} \to \Omega^{p,k-1}} \cdot \Vert B_{t_2- t_1}^k \Vert _{\Omega^{p,k} \to \Omega^{p, k-1}} \to 0$$
when $t_1, t_2$ tend to $+\infty$. Thus the claim follows from completeness.

Now let $\omega \in \Omega^{p, k}(M) \cap \ker d$. From the homotopy relations in the previous lemma, we have for every $t \geqslant 0$
$$d B_t^k (\omega) = \varphi _t^* (\omega) - \omega .$$
By letting $t \to +\infty$, we obtain the following relation in $\Omega^{p, k}(M)$ 
$$ d B_\infty^k (\omega) = - \omega.$$
Therefore $\omega \in d\Omega^{p,k-1}$ and thus $L^p \mathrm{H}_{\mathrm{dR}}^k (M) = \{0\}$.
\end{proof}

\subsection{Identification}

We relate the $L^p$-cohomology of $M$ with the cohomology of certain complexes of currents (Proposition \ref{current-proposition}). 
In some cases, this will lead to non-vanishing
cohomology. See \cite{DS} for a nice introduction to the theory of currents.

\noindent \textit{Another point of view for $\Omega^{p, *}(M)$}. 
Let $D = \dim(M)$. Let $\Omega _c^k (M)$ be the space of compactly supported $C^\infty$ differential $k$-forms, endowed with the $C^\infty$ topology. A \emph{$k$-current}
on $M$ is by definition a continuous real valued linear form on $\Omega _c^{D-k} (M)$. We denote by $\mathcal D '^k(M)$ the space of $k$-currents on $M$ endowed with the weak*-topology. 

The 
\textit{differential} of a $k$-current $T$ is the $(k+1)$-current $dT$ defined
by $dT(\alpha):= (-1)^{k+1} T(d\alpha)$, for every $\alpha \in 
\Omega _c^{D-k-1} (M)$. It defines a map $d$ satisfying $d \circ d =0$.

To every $\omega \in L^p \Omega^k(M)$, one associates the $k$-current $T_\omega$
defined by $T_\omega (\alpha):= \int _M \omega \wedge \alpha$.
The \textit{differential in the sense of currents} of $\omega \in L^p \Omega^k(M)$ is the $(k+1)$-current 
$d\omega:= dT_\omega$.  
One says that $d\omega$ \textit{belongs to $L^p \Omega^{k+1}(M)$} if there exits 
$\theta \in L^p \Omega^{k+1}(M)$ such $d\omega = T_\theta$. These definitions are consistent with the Stokes formula:
$$\int _M d\omega \wedge \alpha = (-1)^{k+1} \int _M \omega \wedge d \alpha.$$

We will use the following caracterisation of the space $\Omega^{p,*}(M)$:

\begin{lemma}
\label{omega-lemma}
The space $\Omega^{p,k} (M)$ is equal to the subspace 
of $L^p \Omega^k (M)$ consisting 
of the $L^p$ $k$-forms whose differentials in the sense of currents belong to $L^p \Omega^{k+1} (M)$. 
Moreover the differential operator $d$ on $\Omega^{p,*} (M)$ agrees with the differential in the sense of currents.
\end{lemma}
 
\begin{proof} Since convergence in $L^p$ implies convergence in the sense of currents, $\Omega^{p,k} (M)$ is contained in space of $L^p$ $k$-forms whose differentials in the sense of currents belong
to $L^p$. The proof of the reverse inclusion is based on a regularization process: according to \cite[Theorem 12.5]{GT06}
there exists a family of operators $R _\varepsilon$, such that for every $L^p$ $k$-form $\omega$ whose differential in the sense of currents belongs
to $L^p$, one has:
\begin{enumerate}
\item $R _\varepsilon \omega$ is a $C^\infty$ $k$-form on $M$,
\item $dR _\varepsilon \omega = R _\varepsilon d\omega$,
\item $\Vert R _\varepsilon \omega - \omega \Vert _{L^p \Omega^k}$ and 
$\Vert R _\varepsilon d\omega - d\omega \Vert _{L^p \Omega^{k+1}}$ tends to $0$ when $\varepsilon \to 0$.
\end{enumerate}
Theses properties imply that $\omega$ belongs to $\Omega^{p,k} (M)$.
\end{proof}

\noindent \textit{The complexes $\Psi^{p, *}(M)$ and $\Psi^{p , *}(M, \xi)$}. Following \cite{Pa09}, 
we introduce two complexes of currents. 
\begin{definition}
For $p \in (1, +\infty)$ and $k \in \bf N$, let $\Psi^{p,k} (M)$ be the space of $k$-currents
$\psi \in \mathcal D '^k(M)$ that can be written
$\psi = \beta + d\gamma$, with $\beta \in L^p \Omega^k (M)$ and $\gamma \in L^p \Omega^{k-1} (M)$. 
In particular $\Psi^{p,0} (M) = L^p(M)$.
For $\psi \in \Psi^{p,k} (M)$, let 
\begin{align*} 
\Vert \psi \Vert _{\Psi^{p,k}} = \inf \Bigl\{&\Vert \beta \Vert _{L^p \Omega^k} + \Vert \gamma \Vert _{L^p \Omega^{k-1}}: 
\psi = \beta + d\gamma, \\
& ~~\mathrm{with}~~ \beta \in L^p \Omega^k (M)~~\mathrm{and}~~ \gamma \in L^p \Omega^{k-1} (M)\Bigr\}.
\end{align*}
\end{definition}

\begin{lemma}\label{current-lemma1}
$\Vert \cdot \Vert _{\Psi^{p,k}}$ is a norm on $\Psi^{p,k} (M)$. With this norm
$\Psi^{p,k} (M)$ is a Banach space. The inclusion maps $L^p\Omega^k (M) \subset \Psi^{p,k} (M) \subset \mathcal D '^k(M)$ are continuous.
The differentials in the sense of currents induce continuous operators
$d_k: \Psi^{p,k} (M) \to \Psi^{p,k+1} (M)$,
that satisfy $d \circ d =0$.
\end{lemma}

\begin{proof}
To see that $\Vert \cdot \Vert _{\Psi^{p,k}}$ is a norm, we only need to check that $\Vert \psi \Vert _{\Psi^{p,k}}=0$ implies $\psi =0$. 
This is a consequence of the following inequality. Let $\psi = \beta + d\gamma \in \Psi^{p,k} (M)$, then by H\"older
inequality one has 
for every $\alpha \in \Omega _c^{D -k-1} (M)$:
\begin{align*}
\vert \psi (\alpha) \vert &= \Bigl\vert \int _M \beta \wedge \alpha + (-1)^{k+1}\int _M \gamma \wedge d\alpha \Bigr\vert\\ 
&\leqslant 
\Vert \beta \Vert _{L^p \Omega^k} \cdot \Vert \alpha \Vert _{L^q \Omega^{D- k}}
+ 
\Vert \gamma \Vert _{L^p \Omega^{k-1}} \cdot \Vert d\alpha \Vert _{L^q \Omega^{D- k +1}},
\end{align*}
where $1/p + 1/q =1$. It also shows that the inclusion map 
$\Psi^{p,k} (M) \subset \mathcal D '^k(M)$ is continuous. The continuity of $L^p\Omega^k (M) \subset \Psi^{p,k} (M)$
is obvious.
To show the completeness, one notices that $\Psi^{p,k}(M)$ is isometric to $L^p \Omega^k(M) \times L^p \Omega^{k-1}(M) / E$, where $E$ is the subspace 
$$E := \{(\beta, \gamma) ~\vert ~ \beta + d \gamma =0\}.$$
Since $\Vert \cdot \Vert _{\Psi^{p,k}}$ is a norm,
$E$ is a closed subspace. Since $L^p \Omega^k(M) \times L^p \Omega^{k-1}(M)$ is a Banach space, the quotient space is Banach too.
Finally the last statement in the lemma is obvious.
\end{proof}

Let $\xi$ be a $C^\infty$ unitary complete vector field on $M$, and denote its flow by $\varphi_t$. 
We suppose that it induces a bounded linear operator 
$\varphi _t^*: L^p \Omega^k (M) \to L^p \Omega^k(M)$, for every $t \in \bf R$ and $k \in \bf N$.
Then it induces an automorphism of the complex $\Psi^{p,*}(M)$ 
whose operator norm satisfies 
$$\Vert \varphi _t^* \Vert _{\Psi^{p,k} \to \Psi^{p,k}} \leqslant \max _{i = k-1, k} \Vert \varphi _t^* \Vert _{L^p \Omega^{i} \to L^p \Omega^{i}}.$$


\begin{definition} For $p \in (1, +\infty)$ and $k \in \bf N$, we set 
$$\Psi^{p,k} (M, \xi)= \{\psi \in \Psi^{p,k} (M): \varphi _t^* (\psi) = \psi~~ \mathrm{for~~ every}~~ t \in \bf R\}.$$
The complex $\Psi^{p,*} (M, \xi)$ is a closed subcomplex of $\Psi^{p,*}(M)$. Let 
$$\mathcal Z^{p,k} (M, \xi) = \Ker \bigl(d: \Psi^{p,k} (M, \xi) \to \Psi^{p,k+1} (M, \xi)\bigr)$$
be the space of $k$-cocycles.
\end{definition}

The next proposition elaborates on \cite[Proposition 10]{P1}. 
It is also a special case of \cite[Corollary 12]{Pa09}.

\begin{proposition}\label{current-proposition} Let $p \in (1, +\infty)$ and $k \in \bf N$. 
Suppose that there exists $C, \eta >0$ such that for every $t \geqslant 0$,
one has 
$$\max _{i = k-1, k} \Vert \varphi _t^* \Vert _{L^p \Omega^{i} \to L^p \Omega^{i}} \leqslant C e^{-\eta t}.$$
(When $i=-1$, we set $L^p \Omega^{-1}:= L^p \Omega^{0}$ for convenience.)
Then there is a canonical Banach isomorphism 
$$L^p \mathrm{H}_{\mathrm{dR}}^{k+1} (M) \simeq \mathcal Z^{p,k+1} (M, \xi),$$
in particular $L^p \mathrm{H}_{\mathrm{dR}}^{k+1} (M)$ is Hausdorff.
\end{proposition}

\begin{proof} The proof is divided into several steps.

(1) We define a bounded linear map 
$$P: \Omega^{p, k+1}(M) \cap \Ker d \to \mathcal Z^{p,k+1} (M, \xi)$$ 
as follows. 
Consider the operators $B^* _t$ defined in Lemma \ref{vanishing-lemma}. 
We claim that the operators $B _t^{k+1}: L^p \Omega^{ k+1}(M) \to L^p \Omega^k(M)$ converge in norm to an operator
$B _\infty^{k+1}$, when $t$ tends to $+\infty$. 
Indeed from its definition one has 
$$\Vert B_t^{k+1} \Vert _{L^p \Omega^{k+1} \to L^p \Omega^k} \leqslant \int _0^t \Vert \varphi^* _s \Vert _{L^p \Omega^k \to L^p \Omega^k} ds.$$
Thus our assumption implies that $B_t^{k+1}$ is bounded independently of $t \geqslant 0$.
Moreover one has $B_{t_2}^{k+1} - B_{t_1}^{k+1}= \varphi^* _{t_1} \circ B_{t_2- t_1}^{k+1}$ for $0\leqslant t_1 \leqslant t_2$.
With our assumption we get that:
\begin{align*}
\Vert B_{t_2}^{k+1} &- B_{t_1}^{k+1} \Vert_{ L^p \Omega^{ k+1} \to L^p \Omega^k} \\
&\leqslant 
\Vert \varphi _{t_1}^* \Vert _{L^p \Omega^k \to L^p \Omega^k} \cdot \Vert B_{t_2-t_1}^{k+1} \Vert _{ L^p \Omega^{ k+1} \to L^p \Omega^k}
\to 0
\end{align*}
when $t_1, t_2 \to +\infty$. The claim follows from completeness.

Let $\omega \in \Omega^{p, k+1}(M) \cap \Ker d$.
For $t \geqslant 0$ one has thanks to Lemma \ref{vanishing-lemma}:
$$\varphi_t^* (\omega) = \omega + dB _t^{k+1} (\omega).$$ 
By letting $t \to +\infty$, and by using the fact that $\Omega^{p, k+1}(M) \cap \Ker d \subset L^p \Omega^{ k+1}(M)$, the claim implies that $\varphi_t^* (\omega)$ converges in $\Psi^{p, k+1}(M)$ to 
the current $\omega _\infty:= \omega + d B_\infty^{k+1} (\omega)$. 

We set $P(\omega):= \omega _\infty$. Since the maps $d$ and $\varphi^*$ are continuous on 
$\Psi^{p, k+1} (M)$, one has $d \omega _\infty =0$ and $\varphi _t^* (\omega _\infty) = \omega _\infty$. 
Therefore $P(\omega)$ belongs to $\mathcal Z^{p,k+1} (M, \xi)$. Moreover $P$ is continuous; indeed 
$$\Vert P \Vert \leqslant 1+ \Vert B^{p, k+1} _\infty \Vert_{L^p \Omega^{k+1} \to L^p \Omega^k}.$$

(2) $P(\omega) =0$ implies $[\omega] =0$ in $L^p \mathrm{H}_{\mathrm{dR}}^{k+1} (M)$:
If $\omega _\infty =0$, one has $\omega = - d B_\infty^{k+1} (\omega)$ 
with $B_\infty^{k+1} (\omega) \in L^p \Omega^{k}(M)$. Therefore by Lemma \ref{omega-lemma} one has $[\omega] =0$.

(3) $P$ is surjective: Let $\psi \in \mathcal Z^{p,k+1} (M, \xi)$. Then there exists $\beta \in L^p \Omega^{k+1}(M)$
and $\gamma \in L^p \Omega^k (M)$ such that $\psi = \beta + d\gamma$ and $d\beta = 0$. By Lemma \ref{omega-lemma}, one has $\beta \in \Omega^{p,k+1}(M) \cap \Ker d$, and:
$$\beta _\infty = \beta + d B_\infty^{k+1} (\beta) = \psi + d \bigl(B_\infty^{k+1} (\beta) -\gamma\bigr).$$
Since $\psi$ and $\beta _\infty$ are both $\varphi _t^*$-invariant, so is $d (B_\infty^{k+1} (\beta) -\gamma)$.
But by our assumption, one has $\Vert \varphi _t^* \Vert _{\Psi^{p,k} \to \Psi^{p,k}} \leqslant C e^{-\eta t}$, and thus:
$$\varphi_t^* \bigl(d (B_\infty^{k+1} (\beta) -\gamma)\bigr) = d\bigl(\varphi_t^* (B_\infty^{k+1} (\beta) -\gamma)\bigr) \to 0$$
when $t \to +\infty$.
Therefore we get that $d (B_\infty^{k+1} (\beta) -\gamma) =0$, and so $\psi = \beta _\infty \in \im P$.

(4) $[\omega] =0$ implies $P(\omega) =0$:
Suppose that $\omega = d \alpha$ with $\alpha \in \Omega^{p,k}(M)$. Then 
$$\omega _\infty = d\alpha + d B_\infty^{k+1} (\omega) = d\bigl(\alpha + B_\infty^{k+1} (\omega)\bigr),$$
with $\alpha + B_\infty^{k+1}(\omega) \in \Psi^{p,k} (M)$. The argument in paragraph (3) yields that 
$d(\alpha + B_\infty^{k+1} (\omega)) = 0$. Thus $\omega_\infty =0$.

(5) Conclusion: Items (1) and (4) show that $P$ induces a continuous homomorphism from $L^p \mathrm{H}_{\mathrm{dR}}^{k+1} (M)$ to $\mathcal Z^{p,k+1} (M, \xi)$; and items (2) and (3) that it is bijective. In particular $L^p \mathrm{H}_{\mathrm{dR}}^{k+1} (M)$ is
Hausdorff and thus a Banach space. By Lemma \ref{current-lemma1}, the space $\mathcal Z^{p,k+1} (M, \xi)$ is Banach too. Therefore Banach Theorem completes the proof.
\end{proof}

\subsection{Non-vanishing} 
The next proposition is a non-vanishing result which complements Proposition \ref{current-proposition}. 
It is a special case of \cite[Corollary 32]{Pa09}.

Observe that the space $\Ker (\iota _\xi: L^p \Omega^* (M) \to L^p \Omega^{* -1}(M))$ is $\varphi_t^*$-invariant, since $\iota _\xi$ and $\varphi_t^*$ commute.
\begin{proposition}\label{current-proposition2} 
Let $p \in (1, +\infty)$ and $k \in \{0, \dots, D-2\}$ where $D = \dim M$. 
Suppose that $(\varphi _t) _{t \in{\bf R}}$ acts properly on $M$ and 
that there exists
$C, \eta > 0$ such that for $t \geqslant 0$:
\begin{align*}
&\Vert \varphi^* _t \Vert _{L^p \Omega^k \to L^p \Omega^k} \leqslant C e^{-\eta t},\\
\mathrm{and}~~~~~~ &\Vert \varphi^* _{-t} \Vert _{L^p \Omega^{k +1} \cap \Ker \iota _\xi \to L^p \Omega^{k+1} \cap \Ker \iota _\xi} \leqslant C e^{-\eta t}.
\end{align*}
(Note the opposite signs in front of $t$ in the left side members). 
Then $\mathcal Z^{p,k+1} (M, \xi)$ is non-trivial.
\end{proposition}

\begin{proof} Since $(\varphi _t) _{t \in{\bf R}}$ acts properly on $M$, every orbit admits an invariant open neighborhood 
$U \simeq {\bf R} \times V$ on which $\varphi_t$ acts like a translation along the $\bf R$ factor.
In the sequel we suppose that such an orbit neighborhood has been choosen.

Pick a non-trivial $\varphi_t$-invariant form $\alpha \in \Omega^k (M) \cap \Ker \iota _\xi$ which is supported in $U$. Since $k \in \{0, \dots, D-2\}$, we can choose $\alpha$ so that $d \alpha \neq 0$. Set $\psi:= d\alpha$.
We claim that $\psi \in \mathcal Z^{p,k+1} (M, \xi)$. 

Let $\chi$ be a $C^\infty$ function on $U \simeq {\bf R} \times V$,
depending only on the $\bf R$-variable, and such that $\chi (t) = 0$ for $t \leqslant 0$ and $\chi (t) =1$ for $t \geqslant 1$. 
One has $\alpha = \chi \cdot \alpha + (1 - \chi)\cdot \alpha$, and thus:
$$\psi = \beta + d\gamma ~~\mathrm{with}~~ \beta = d(\chi \cdot \alpha) = d\chi \wedge \alpha + \chi \cdot d\alpha 
~~\mathrm{and}~~ \gamma = (1 - \chi)\cdot \alpha.$$
The form $d\chi \wedge \alpha$ has compact support. The form $d \alpha$ belongs to $\Ker \iota _\xi$, thanks to the formula 
$\mathcal L _\xi = d \circ \iota _\xi + \iota _\xi \circ d$. Since the restriction of 
$d\alpha$ to $[0,1] \times V$ belongs to $L^p$, the assumption on the norm of $\varphi_{-t}^*$ implies that the restriction of $d\alpha$ to $[0, +\infty) \times V$ belongs to $L^p$ too.
Therefore $\beta \in L^p \Omega^{k+1} (M)$. 

Since the restriction of $\alpha$ to $[0,1] \times V$ belongs to $L^p$, the assumption on the norm of $\varphi _t^*$ implies that the restriction 
of $\alpha$ to $(-\infty, 0]\times V$
belongs to $L^p$ too. Therefore $\gamma \in L^p \Omega^k(M)$.
\end{proof}


\subsection{Poincar\'e duality} Parts of vanishing and non-vanishing results for $L^p$-cohomology 
rely on the following version of Poincar\'e duality. 

\begin{proposition}\label{poincare-proposition} Let $M$ be a complete oriented Riemannian manifold of dimension $D$.
Let $p \in (1, +\infty)$, $q = p/(p-1)$ and $0 \leqslant k \leqslant D$. Then 
\begin{itemize}
\item $L^p {\rm H}_{\mathrm{dR}}^k (M)$ is Hausdorff if and only if $L^q {\rm H}_{\mathrm{dR}}^{D-k +1} (M)$ is.
\item $L^p \overline{{\rm H}^k _{\mathrm{dR}}} (M)$ and $L^q \overline{{\rm H}^{D-k}_{\mathrm{dR}}} (M)$ are dual Banach spaces.
Thus $L^p \overline{{\rm H}^k _{\mathrm{dR}}} (M) = \{0\}$ if and only if $L^q \overline{{\rm H}^{D-k}_{\mathrm{dR}}} (M) = \{0\}$.
\end{itemize}
\end{proposition}
\begin{proof} See \cite[Corollaire 14]{P1} or \cite{GT10}.
\end{proof}

\subsection{Examples and complements}\label{example-complement}
{\bf {1) }} 
Let $M$ be a complete connected Riemannian manifold of infinite volume and dimension $D$, 
and let $p \in (1, +\infty)$, $q = p/(p-1)$.
It is obvious that $L^p {\rm H}_{\mathrm{dR}}^0 (M) = \{0\}$. By Poincar\'e duality (Proposition \ref{poincare-proposition})
this implies that $L^q \overline{{\rm H}^D _{\mathrm{dR}}} (M) = \{0\}.$
In another hand, again by Poincar\'e duality, $L^q {\rm H}_{\mathrm{dR}}^D (M)$ is Hausdorff if and only if 
$L^p {\rm H}_{\mathrm{dR}}^1 (M)$ is. When $M$ admits a complete unit vector field whose flow expands exponentially the volume of $M$,
then $L^p {\rm H}_{\mathrm{dR}}^1 (M)$ is Hausdorff for every $p \in (1, +\infty)$. Indeed this follows from Proposition
\ref{current-proposition} applied with $k=0$. This shows in particular that any connected non-unimodular Lie group $G$
of dimension $D$ satisfies $L^q {\rm H}_{\mathrm{dR}}^D(G) = \{0\}$ for every $q \in (1, +\infty)$. In fact, Pansu in \cite[Th\'eor\`eme 1]{P2} shows that for any 
non-compact connected Lie group $G$, either  $L^p {\rm H}_{\mathrm{dR}}^1(G)$ is Hausdorff 
or $G$ is amenable unimodular. See also \cite{T} for generalizations to topological groups.

{\bf {2) }} 
Under assumptions of Proposition \ref{current-proposition}, one has 
$\Vert \varphi _t^* \Vert _{\Psi^{p,k} \to \Psi^{p,k}} \to 0$ when $t \to +\infty$, and thus
$\Psi^{k,p}(M, \xi) = \{0\}$.
Therefore the conclusion of Proposition \ref{current-proposition} can be stated as
$L^p \mathrm{H}_{\mathrm{dR}}^{k+1} (M) \simeq \mathrm{H}^{k+1} (\Psi^{p,*}(M, \xi))$.
In \cite[Theorem 4]{Pa09}, Pansu establishes a general result that relates the $L^p$-cohomology of $M$ 
with the cohomology of $\Psi^{p,*}(M, \xi)$.

{\bf{3)}} Suppose $M$ is diffeomorphic to ${\bf R} \times N$ and that $\xi = \frac{\partial}{\partial t}$ is the unit vector field carried by the ${\bf R}$-factor. Let $\pi: M \to N$ be the projection,
and let $\pi^*: \mathcal D'^i(N) \to \mathcal D'^i(M)$ be the continuous extension of the pull-back map $\pi ^* : \Omega ^i(N) \to \Omega ^i (M)$. Then, under the assumptions of Proposition
\ref{current-proposition}, one can prove that every $\psi \in \mathcal Z^{p,k+1}(M, \xi)$ can be
written as $\psi = \pi^*(T)$, with $T \in \mathcal D'^{k+1}(N) \cap \Ker d$. In other words, one has 
$$\mathcal Z^{p,k+1}(M, \xi) = \pi^* \bigl(\mathcal D'^{k+1}(N) \cap \Ker d\bigr) \cap \Psi^{p, k+1}(M).$$
Therefore, the isomorphism $L^p \mathrm{H}_{\mathrm{dR}}^{k+1} (M) \simeq \mathcal Z^{p,k+1} (M, \xi)$ can be interpreted by viewing $\mathcal Z ^{p,k+1}(M, \xi)$ as a ``boundary values'' space for the 
cohomology classes in $L^p \mathrm{H}_{\mathrm{dR}}^{k+1} (M)$. See \cite{Pa09} for more information along these lines.

\section{The Lie group case}
\label{s - Lie group dR}

This section collects applications of the previous results to Lie groups. 
It will serve as a main tool in the paper. 

Let $G$ be a connected Lie group equipped with a left-invariant Riemannian metric. Let $D = \dim(G)$ and let $\mathfrak g$ 
be its Lie algebra. 
As usual the left and right multiplications by $g \in G$ are denoted $L_g$ and $R_g$,
and we let $C_g = L_g \circ R_{g^{-1}}$ be the conjugacy by $g$.
Given a unit vector $\xi \in \mathfrak g$, 
we will still denote by $\xi$ the associated left-invariant vector field on $G$. 
Its flow is $\varphi _t = R_{\exp t \xi}$.

The following result is implicit in \cite{Pa99}, see in particular \cite[Corollaire 53]{Pa99} and  \cite[Proposition 57]{Pa99}. 

\begin{proposition}
\label{lie-proposition} 
Let $\xi \in \frak g$ and suppose that its flow acts properly on $G$. 
Let $\lambda _1 \leqslant \lambda _2 \leqslant \dots \leqslant \lambda _D$ 
be the real parts of the eigenvalues of $(-\mathrm{ad}\xi) \in \mathrm{End}(\mathfrak g)$, 
enumerated with their multiplicities in their generalized eigenspaces. We denote by $w_k = \sum _{i=1}^{k+1} \lambda _i$ 
the sum of the $(k+1)$ first ones, and by $W_k = \sum _{j=1}^k \lambda _{D-j+1}$ 
the sum of the $k$ last ones. 
Let $h = \sum _{i=1}^D \lambda _i$ be the trace of $(-\mathrm{ad} \xi)$, and suppose that $h>0$. Let $k \in \{0, \dots, D-1\}$.
\begin{enumerate}
\item If $1< p< \frac{h}{W_k}$, then $L^p \mathrm{H}^k _{\mathrm{dR}}(G) = \{0\}$.
\item Under the same assumptions, the space $L^p \mathrm{H}^{k+1} _{\mathrm{dR}}(G)$ is 
Hausdorff and Banach isomorphic to $\mathcal Z^{p,k+1}(M, \xi)$.
\item If $1 \leqslant \frac{h}{w_{k+1}}<p<\frac{h}{W_k}$, then 
$L^p \mathrm{H}^{k+1} _{\mathrm{dR}}(G) \neq \{0\}$.
\end{enumerate}
(When $k=0$, we set $W_0 :=0$ and $\frac{h}{W_0} := +\infty$.)
\end{proposition}

\begin{remark} 
The above proposition applies to non-unimodular Lie groups only (but no solvability or semi-simplicity assumption is made). 
Indeed the trace of $\mathrm{ad} \xi$ is supposed to be non-zero.
\end{remark}

\begin{proof} (1) and (2). To obtain vanishing and identification of the cohomology from Propositions \ref{vanishing-proposition} and \ref{current-proposition}, we need to control the norm of
$\varphi^* _t: L^p \Omega^*(G)
\to L^p \Omega^*(G)$, for $t \geqslant 0$.
Let $k \in \{0, \dots, D-1\}$ and $\omega \in \Omega^k(G)$. Suppose first that $\omega$ is left-invariant. Then $\varphi^* (\omega)$
is left-invariant too. In particular their norms $\vert \omega \vert _g$ and $\vert \varphi _t^* (\omega) \vert _g$ are independant of $g \in G$.
We have 
$$ \varphi _t^* (\omega) = R_{\exp t \xi}^* (\omega)
= L^* _{\exp -t\xi}(R_{\exp t \xi}^* (\omega)) = C _{\exp -t\xi}^* (\omega). $$
Moreover for $X_1, \dots, X_k \in \mathfrak g$, we have
\begin{align*}
C _{\exp -t\xi}^* (\omega) (X_1, \dots, X_k)& 
= \omega \bigl(\mathrm{Ad} (\exp -t \xi) X_1, \dots, \mathrm{Ad} (\exp -t \xi) X_k\bigr)\\
& = \omega (e^{-t \mathrm{ad} \xi} X_1, \dots, e^{-t \mathrm{ad} \xi} X_k).
\end{align*}
In other words, by identifying the space of left-invariant $k$-forms with $\Lambda^k \mathfrak g^*$, we get that
$C _{\exp -t\xi}^* (\omega) = (e^{-t \mathrm{ad} \xi})^* (\omega).$ 

Observe that $W_k$ is the largest real part 
of the eigenvalues of $(-\mathrm{ad} \xi)^*$ acting on $\Lambda^k \mathfrak g^*$. 
Thus for every $\varepsilon >0$ there is a constant $C = C(\varepsilon, k) >0$ such that
for all $\omega \in \Lambda^k \mathfrak g^*$ and $t\ge0$, one has
$$\vert \varphi _t^* (\omega) \vert \leqslant C e^{(W_k + \varepsilon)t} \vert \omega \vert.$$
Now let $\omega$ be any element of $\Omega^k(G)$ (it is not supposed to be left-invariant anymore). 
Pick a basis $(\theta _I)$ of $\Lambda^k \mathfrak g^*$ seen as the space of left-invariant $k$-forms on $G$.
Then $\omega$ decomposes uniquely as $\omega = \sum _I f_I \theta _I$, where $f_I \in \Omega^0(G)$. 
Since the norms on $\Lambda^k \mathfrak g^*$ are all equivalent, there exists a constant $D = D(k)>0$ such that for every 
$\omega \in \Omega^k(G)$ and $g \in G$: 
$$D^{-1} \bigl( \sum _I \vert f_I(g) \vert^p \bigr) ^{1/p}
\leqslant \vert \omega \vert _g \leqslant 
D \bigl( \sum _I \vert f_I(g) \vert^p \bigr) ^{1/p}.$$
On the other hand, for $t \geqslant 0$: 
$$\vert \varphi_t^* (\omega) \vert _g = \bigl\vert \sum _I f_I \circ \varphi _t\cdot \varphi _t^* (\theta _I) \bigr\vert _g
\leqslant C \sum _I \vert (f_I \circ \varphi _t)(g) \vert  e^{(W_k + \varepsilon)t}.$$
Therefore: 
\begin{align*}
\Vert \varphi _t^* (\omega) \Vert _{L^p \Omega^k}^p &= \int _G \vert \varphi _t^* (\omega) \vert _g^p ~d\mathrm{vol}(g)\\
&\leqslant C^p e^{p(W_k + \varepsilon)t} \int_G \sum _I \vert (f_I \circ \varphi _t)(g) \vert^p ~d \mathrm{vol}(g)\\
&= C^p e^{p(W_k + \varepsilon)t} \int_G \sum _I \vert f_I (h) \vert^p \mathrm{Jac}( \varphi _t^{-1}) (h) ~d \mathrm{vol} (h)\\
&= C^p e^{p(W_k + \varepsilon - \frac{h}{p})t} \int_G \sum _I \vert f_I (h) \vert^p ~d \mathrm{vol} (h)\\
&\leqslant C ^p D^p e^{p(W_k + \varepsilon - \frac{h}{p})t}  \Vert \omega \Vert _{L^p \Omega^k}^p,
\end{align*}
since the Jacobian of $\varphi_t$ is $e^{ht}$.
Finally, by letting $E := CD$, one obtains that 
$$\Vert \varphi _t^*\Vert _{L^p \Omega^k \to L^p \Omega^k} \leqslant E e^{(W_k + \varepsilon - \frac{h}{p})t}.$$

Suppose that $k \geqslant 1$ and $1<\frac{h}{W_k}$. 
Then, just from their definitions, one sees that $0 \leqslant W_{k-1}\leqslant W_k < h$ (recall that $h>0$ by assumption). This implies that
$W_k + \varepsilon - \frac{h}{p} <0$ for $1<p < \frac{h}{W_k}$ and $\varepsilon >0$ small enough; and that the same holds with $W_{k-1}$ instead of $W_k$.
We conclude then by using Propositions \ref{vanishing-proposition} and \ref{current-proposition}. 

Suppose that $k =0$. Then $W_0 + \varepsilon - \frac{h}{p} = \varepsilon - \frac{h}{p}<0 $ for $\varepsilon >0$ small enough. Thus we can conclude as well by using Propositions \ref{vanishing-proposition} and \ref{current-proposition}.

(3) To prove non-vanishing of the cohomology by using Propositions \ref{current-proposition} and \ref{current-proposition2}, we need in addition
to control the norm of 
$$\varphi^* _{-t}: L^p \Omega^{k+1}(G) \cap \Ker \iota _\xi
\to L^p \Omega^{k+1}(G) \cap \Ker \iota _\xi,$$ 
for $t\geqslant 0$.

The group $(\varphi_{-t})_{t \in{\bf R}}$ is the flow of the vector field $-\xi$. The real parts of the eigenvalues of $\mathrm{ad}\xi$ are
$-\lambda _D \leqslant -\lambda _{D-1} \leqslant \dots \leqslant -\lambda _1$. The sum of the $(k+2)$ last ones is $-w_{k+1}$; and the trace of 
$\mathrm{ad}\xi$ is $-h$. We notice that $\xi$ is an eigenvector of eigenvalue $0$. By assumption one has $w_{k+1} >0$; thus the eigenvalue
$0$ appears in the list $\lambda _1 \leqslant \dots \leqslant \lambda _{k+2}$, whose sum is $w_{k+1}$. Therefore $-w_{k+1}$ 
is largest real part of the eigenvalues of $(\mathrm{ad}\xi)^*$ acting on $(\Lambda^{k+1} \frak g^*) \cap \Ker \iota _\xi$.

By the same argument as in the first part of the proof, one obtains that for every $\varepsilon >0$ there is a constant $E'=E'(\varepsilon, k) >0$, such that for every $t \geqslant 0$:
$$\Vert \varphi^* _{-t} \Vert _{L^p \Omega^{k+1} \cap \Ker \iota _\xi \to L^p \Omega^{k+1} \cap \Ker \iota _\xi} \leqslant E' e^{(-w_{k+1} + \epsilon +\frac{h}{p})t}.$$
For $p > \frac{h}{w_{k+1}} $ and $\varepsilon >0$ small enough, one has $-w_{k+1} + \epsilon +\frac{h}{p}<0$. 
We conclude by using Propositions \ref{current-proposition} and \ref{current-proposition2}.
\end{proof}

\begin{example}
Let $R = {\bf R} \ltimes_\alpha {\bf R}^{D-1}$ with $\alpha: {\bf R} \to \mathrm{GL}({\bf R}^{D-1})$ defined by 
$\alpha (t) = e^t \cdot \mathrm{id}$.
Then $R$ is isometric to the real hyperbolic space $\mathbb H^D _{\bf R}$. Consider the vector field $\frac{\partial}{\partial t}$. The 
linear map $\mathrm{ad}(\frac{\partial}{\partial t})$ is the projection on the ${\bf R}^{D-1}$ factor. 
Put $\xi = -\frac{\partial}{\partial t}$. The list of the eigenvalues of 
$(-\mathrm{ad}\xi)$ is $0= \lambda _1 < \lambda _2 = \dots = \lambda _D =1$. For every 
$k \in \{0, \dots , D-1\}$
one has $w_k = W_k = k$. Moreover $h = D-1$. Therefore the above proposition applied in combination with Poincar\'e duality (Proposition \ref{poincare-proposition}), shows that for $k\in \{1, \dots , D-1\}$: 
\begin{itemize}
\item $L^p \mathrm{H}^k _{\mathrm{dR}}(R) =\{0\}$ for $1< p<\frac{D-1}{k}$ or $p>\frac{D-1}{k-1}$. 
\item If $\frac{D-1}{k}<p <\frac{D-1}{k-1}$, then $L^p \mathrm{H}^{k} _{\mathrm{dR}}(R) \neq \{0\}$ and is Banach isomorphic to 
$\mathcal Z^{p,k}(M, \xi)$.
\end{itemize}
Moreover, by Subsection \ref{example-complement}.1, one has $L^p \mathrm{H}^0 _{\mathrm{dR}}(R) = L^p \mathrm{H}^D _{\mathrm{dR}}(R) =\{0\}$ for every $p \in (1, +\infty)$.
\end{example}

Apart for the above example, and few other ones, the $L^p$-cohomology is only partially known.
In the sequel we will use the following consequences of Proposition \ref{lie-proposition}.

\begin{corollary} 
\label{lie-corollary1} 
Suppose that $R$ is a connected Lie group of the form $R = A \ltimes N$, with $A \simeq {\bf R}^l$ and $l \geqslant 1$, that satisfies the contraction condition {\rm (nC)} of the introduction (for a given $\xi$ in the Lie algebra of $A$). Let $D =\dim(R)$.
Then, with the notations of Proposition \ref{lie-proposition}, the numbers $w_{l}$, $W_{D-1-l}$ belongs to $(0, h)$, and we have:
\begin{itemize}
\item If $1<p<\frac{h}{W_{D-1-l}}$, then $L^p \mathrm{H}^k _{\mathrm{dR}}(R) = \{0\}$ for all $k < D-l$.
\item Under the same assumption, the space $L^p \mathrm{H}^{D-l} _{\mathrm{dR}}(R)$ is Hausdorff and Banach isomorphic to $\mathcal Z^{p, D-l}(R, \xi)$. 
\item If $p > \frac{h}{w_{l}}$, then $L^p \mathrm{H}^k _{\mathrm{dR}}(R) =\{0\}$ for all $k > l$ .
\end{itemize}
\end{corollary}

\begin{proof} By condition {\rm (nC)}, the list of the real parts of the eigenvalues of $(-\mathrm{ad}\xi)$ is of the form
$$0=\lambda _1 = \dots = \lambda _l < \lambda _{l+1} \leqslant \dots \leqslant \lambda _D.$$
Thus, one has $0< W_k \leqslant W_{D-1-l} <h$ for every $k <D-l$. 
By Proposition \ref{lie-proposition} the first two items follow. The last one is a consequence of the previous ones in combination with Poincar\'e duality (Proposition \ref{poincare-proposition}). It holds 
for $p$ larger than the H\"older conjugate of 
$\frac{h}{W_{D-l-1}}$, \emph{i.e.}~ for $p> \frac{h}{w_{l}}$
(since $W_{D-l-1} + w_{l}=h$). Note that $w_l = \lambda_{l+1}$.
\end{proof}

In the special case $l=1$, one gets in addition:

\begin{corollary} 
\label{lie-corollary2} 
Let $R$ be as in the previous corollary, and assume in addition that $l=1$.
If $1<p < \frac{h}{W_{D-2}}$, then $L^p \mathrm{H}^{D-1} _{\mathrm{dR}}(R)$ is non-trivial. 
\end{corollary}

\begin{proof} With the notations of Proposition \ref{lie-proposition}, we have: $0< W_{D-2} <w_{D-1} <h$. Thus the result follows from Proposition \ref{lie-proposition}, applied with $k= D-2$.
\end{proof}

\section{A non-vanishing criterion}
\label{s - non-vanishing crit} 

Let $R$ be a connected Lie group of the form $R = A \ltimes N$ with $A \simeq {\bf R}^l$ and $l \geqslant 1$, that satisfies conditions (nC) and (nT) of the introduction. 
As explained in Subsection 
\ref{introduction-strategy}, it decomposes as $R = B \ltimes H$ and the analysis of 
$L^p\overline{\mathrm{H}_{\mathrm{dR}}^l}(R)$ -- for large $p$ and $l = \dim(A)$ -- reduces to the study of the action of $B$
on $L^p\mathrm {H}_{\mathrm{dR}}^{D-l} (H)$ for $p$ close to $1$. The goal of this section is to give a sufficient condition
on this action (Proposition \ref{non-vanishing-proposition}), to ensure non-vanishing of 
$L^p\overline{\mathrm{H}_{\mathrm{dR}}^l}(R)$ for large $p$ (Corollary \ref{non-vanishing-corollary}). 

As an application we will prove Theorems \ref{introduction-theorem2}, \ref{introduction-theorem3}
and \ref{introduction-theorem4} in the next two Sections.

We keep the objects and notations of Subsection \ref{introduction-strategy}. Recall that the Lie algebra of $A$ is denoted by
$\frak a$, and that the Lie algebra of $B$ is
$\frak b =\{X \in \frak a: \mathrm{trace}(\mathrm{ad}X) = 0\}$. The group $H$ is defined by by $H = \{e^{t\xi}\}_{t \in \mathbf R} \ltimes N$, where $\xi \in \frak a$ is a vector that satisfies condition (nC). 

As in the statement of Proposition \ref{lie-proposition}, we consider the real parts of the eigenvalues of $(-\mathrm{ad}\xi)$ acting on the Lie of $R$. Since $\xi$ satisfies condition (nC), one has 
$$0=\lambda _1 = \dots = \lambda _l < \lambda _{l+1} \leqslant \dots \leqslant \lambda _D.$$
Recall that $w_k := \sum _{i=1}^{k+1} \lambda _i$ denotes
the sum of the $(k+1)$ first ones, and $W_k := \sum _{j=1}^k \lambda _{D-j+1}$ 
the sum of the $k$ last ones. The trace of $(-\mathrm{ad}\xi)$ is $h := \sum _{i=1}^D \lambda _i$. 

We also set $n:= D-l = \dim(N)$, and we denote by $\pi$ the projection $\pi: H \to N$ and by $\mathrm{vol}$ the Riemannian volume form on $N$. 

\subsection{Interpolation} Our first result is a quantitative version of Proposition \ref{current-proposition2}. 
It shows that the norm $\Vert \cdot \Vert _{\Psi^{p,n}(H)}$ interpolates (in some sense) between 
$\Vert \cdot \Vert _{L^p \Omega^n (N)}$ and $\Vert \cdot \Vert _{L^p \Omega^{n-1} (N)}$.
\begin{lemma}\label{non-vanishing-lemma}
For every $1<p< \frac{h}{W_{n-1}}$, there exist constants
$a, b, C >0$, such that for every $\theta \in \Omega^{p, n-1} (N)$ 
the form $\pi^* (d\theta)$ belongs to $\mathcal Z^{p, n} (H, \xi)$, and satisfies:
$$\Vert \pi^*(d\theta) \Vert _{\Psi^{p, n}(H)} \leqslant C  \inf _{s \geqslant 0} \Bigl\{e^{-as} \Vert d\theta \Vert _{L^p \Omega^n(N)} + e^{bs} \Vert \theta \Vert _{L^p \Omega^{n-1}(N)}\Bigr\}.$$ 
\end{lemma}

\begin{proof} Recall that the flow of $\xi$ in $H$ is denoted by $\varphi _t$. 
It is simply a translation along the ${\bf R}$ factor of $H$.
From Corollary \ref{lie-corollary2} (and its proof), for every $1<p< \frac{h}{W_{n-1}}$ there exist $\eta >0$ and $C_1> 0$ such that for $t \geqslant 0$:
\begin{enumerate}
\item $\Vert \varphi^* _t \Vert _{L^p \Omega^{n-1} \to L^p \Omega^{n-1}} \leqslant C_1 e^{-\eta t}$,
\item $\Vert \varphi^* _{-t} \Vert _{L^p \Omega^{n} \cap \Ker \iota _\xi \to L^p \Omega^{n} \cap \Ker \iota _\xi} 
\leqslant C_1 e^{-\eta t}$.
\end{enumerate}

For $s \in {\bf R}$, let $\chi_s \in C^\infty (H)$ which depends only on the first variable, and such that $\chi_s (t) = 0$ for $t\leqslant s$, and 
$\chi_s (t)=1$ for $t\geqslant s+1$. 
Let $\alpha = \pi^* (\theta)$. Then $\alpha$ and $d\alpha$ are $\varphi _t^*$-invariant. 
One has 
$$\pi^* (d\theta) = d\alpha = d(\chi_s \cdot \alpha + (1-\chi _s)\cdot \alpha) = \beta + d\gamma,$$
with $\beta = d(\chi_s \cdot \alpha) = d\chi_s \wedge \alpha + \chi_s \cdot d \alpha$, and $\gamma = (1- \chi_s)\cdot \alpha$.
Thus $\Vert \pi^* (d\theta) \Vert _{\Psi^{p,n}(H)} \leqslant \Vert \beta \Vert _{L^p \Omega^n (H)} + \Vert \gamma \Vert _{L^p \Omega^{n-1} (H)}$.

The form $d\alpha$ belongs to $\Ker \iota _\xi$ and is $\varphi _t^*$-invariant. With the above property (2) 
we obtain (since $s \geqslant 0$):
\begin{align*}
\Vert \chi_s \cdot d\alpha \Vert _{L^p \Omega^n (H)} 
&\leqslant \Vert (d\alpha)\cdot \mathbf{1}_{t\geqslant s} \Vert _{L^p \Omega^n (H)}\\
&= \sum _{i=0}^\infty  \Vert (d\alpha)\cdot \mathbf{1}_{t\in [s+i, s+i+1]} \Vert _{L^p \Omega^n (H)}\\
&\leqslant C_1 \sum _{i=0}^\infty e^{-\eta (s+i)}\Vert (d\alpha)\cdot \mathbf{1}_{t\in [0,1]} \Vert _{L^p \Omega^n (H)}\\
&= \frac{C_1 e^{-\eta s}}{1 - e^{-\eta}} \Vert (d\alpha)\cdot \mathbf{1}_{t\in [0,1]} \Vert _{L^p \Omega^n (H)}.
\end{align*}
Since the Haar measure on $H$ is the product of those on $\mathbf R$ and $N$ times $e^{-th}$, the $L^p$-norms of 
$(d\alpha)\cdot \mathbf{1}_{t\in [0,1]} $ and of $d\theta$ are comparable.
Therefore there exists a constant $C_2$, depending only on $C_1$, $\eta$ and $H$, such that  
$$\Vert \chi_s \cdot d\alpha \Vert _{L^p \Omega^n (H)} \leqslant C_2 e^{-\eta s} \Vert d\theta \Vert _{L^p \Omega^n(N)}.$$

It remains to bound by above the $L^p$-norms of $d\chi_s \wedge \alpha$ and $\gamma$.
One has $d \chi _s = \chi_s'(t)dt$, with $\chi_s'$ supported on $[s,s+1]$. Thus
$$\Vert d\chi_s \wedge \alpha \Vert_{L^p \Omega^n (H)} \leqslant \Vert \chi_s' \cdot \alpha \Vert _{L^p \Omega^{n-1} (H)} 
\leqslant C_3 \Vert \alpha \cdot  \mathbf{1}_{t\in [0,s+1]}\Vert _{L^p \Omega^{n-1} (H)},$$
with $C_3 = \Vert \chi_s' \Vert _\infty$.

Similarly, by writing  
$\gamma = \alpha \cdot \mathbf 1 _{t \leqslant 0} + (1- \chi_s) \cdot \alpha \cdot \mathbf 1 _{t \in [0, s+1]}$,
one has 
$$\Vert \gamma \Vert _{L^p \Omega^{n-1} (H)} \leqslant \Vert \alpha \cdot  \mathbf{1}_{t \leqslant 0}\Vert _{L^p \Omega^{n-1} (H)}
+ \Vert \alpha \cdot  \mathbf{1}_{t\in [0,s+1]}\Vert _{L^p \Omega^{n-1} (H)}.$$
The form $\alpha$ is $\varphi _t^*$-invariant. By using property (1) and a similar argument as above, one obtains that the $L^p$-norm of $\alpha \cdot \mathbf 1 _{t \leqslant 0}$ is control by above by a constant times 
$\Vert \theta \Vert _{L^p \Omega^{n-1}(N)}$.

Since $\varphi _t$ increases the norms at most exponentially, there exists $b >0$ and $C_4 >0$ such that for $t\geqslant 0$ one has
$$\Vert \varphi _{-t}^* \Vert _{L^p \Omega^{n-1} \to L^p \Omega^{n-1}} \leqslant C_4 e^{b t}.$$
Thus the $L^p$-norm of $\alpha \cdot \mathbf 1 _{t \in [0, s+1]}$ is bounded by above by a constant times 
$e^{b s} \Vert \theta \Vert _{L^p \Omega^{n-1}}$.
The statement follows with $a = \eta$.
\end{proof}

\subsection{The criterion} 
\label{ss - criterium} 
We now state the criterion that will lead to non-vanishing of cohomology. 
It requires some preparations.

Since $\frak a$ is abelian, 
there is a basis of $\frak n \otimes \mathbf C$ relative to which
the matrices of $\mathrm{ad}X \restr _{\frak n}$ ($X \in \frak a$) are upper triangular \cite[4.1. \!Corollary A]{Humphreys-LA}. 
For $i \in \{1, \dots , n\}$, denote by $\varpi_i$ the $i$-th diagonal coefficient; it is a linear form on $A$ with real or complex values. 
Since we assume that condition (nT) holds, we have: 
$$\bigcap _{i =1}^n \Ker \Re (\varpi_i) = \{0\}.$$
By definition of the factor $\frak b < \frak a$, one has $\mathrm{trace}( \mathrm{ad} X) =0$ for every $X \in \frak b$;
and thus $\sum _{i=1}^n \Re (\varpi_i (X)) =0$.
In combination with the above trivial intersection, this implies that for every $X \in \frak b \setminus \{0\}$,
there exists $i \in \{1, \dots , n\}$
such that $\Re (\varpi_i (X)) >0$.

\begin{proposition}
\label{non-vanishing-proposition} 
Suppose that we are given:
\begin{enumerate}
\item A non-empty subset $J \subset \{1, \dots , n\}$, such that for every vector $X \in \frak b \setminus \{0\}$ there exists $j \in J$ 
with $\Re (\varpi_j (X)) >0$.
\item To every $j \in J$, a vector $Z_j \in \frak n$ or $\frak n \otimes \mathbf C$ depending whether $\varpi_j$ takes real or complex values,
such that $\mathrm{ad}(X)\cdot Z_j = \varpi _j (X) Z_j$ for every $X \in \frak b$. 
\item A non-zero $C^1$ function $f: N \to {\bf R}$ with compact support, whose integral along every orbit
of the left-invariant vector fields $Y_j:= \Re Z_j$ $(j \in J)$ is null.
\end{enumerate}
Then for $1<p <\frac{h}{W_{n-1}}$, the form $\psi = \pi^*(f\mathrm{vol})$ belongs to $\mathcal Z^{p, n} (H, \xi)$,
and satisfies:
$$\int _{X \in \frak b} \bigl\Vert C^* _{\exp X} (\psi) \bigr\Vert _{\Psi^{p,n}(H)}^p ~dX <+\infty.$$
\end{proposition}

\begin{remark} 
\label{rk - non-vanishing condition} 
Note that assumption (2) holds for any $J$, since the $w_i$'s are also the weights of the representation $\mathrm{ad}: A \to \mathrm{End}(\frak n \otimes \mathbf C)$.
Moreover, as explained above, assumption (1) is valid with the maximal choice $J = \{1, \dots, n\}$.
In general, condition (1) leads to choosing big subsets $J$ of $\{1, \dots , n\}$ but this balanced by condition (3). 
The latter assumption is \textit{not} a straightforward consequence of our previous hypothesis and the main part of the next section will be dedicated to sufficient conditions for it. 
\end{remark}

\begin{corollary}
\label{non-vanishing-corollary} 
Suppose that the conditions \emph{(1), (2), (3)} above are satisfied. 
Then  
$L^p\overline{\mathrm{H}_{\mathrm{dR}}^l}(R) \neq \{0\}$ for $p >\frac{h}{w_l}$.
\end{corollary}

\begin{proof} Let $1<p<\frac{h}{W_{n-1}}$. According to Corollary \ref{lie-corollary1} applied with $R=H$, the space $L^p\mathrm {H}_{\mathrm{dR}}^{k} (H)$ vanishes for $k<n$, and is Hausdorff and Banach isomorphic to $\mathcal Z^{p, n} (H, \xi)$ when $k=n$. As explained in the sketch of proof in Subsection \ref{introduction-strategy}, these properties in combination with a spectral sequence argument taken from \cite{BR} -- see Corollary \ref{coho-corollary}, imply that there exists a linear isomorphism:
$$
L^p \mathrm{H} _{\mathrm{dR}}^{D-l}(R) \simeq \Bigl\{\psi \in \mathcal Z^{p, n} (H, \xi): \int _{X \in \frak b} \bigl\Vert C^* _{\exp X} (\psi) \bigr\Vert _{\Psi^{p,n}(H)}^p ~dX < +\infty \Bigr\}.
$$
Thus by Proposition \ref{non-vanishing-proposition}, the space $L^p\mathrm {H}_{\mathrm{dR}}^{D-l}(R)$ is non-trivial. Moreover it Hausdorff (see Corollary \ref{lie-corollary1}). By Poincar\'e duality (Proposition \ref{poincare-proposition}) we obtain that $L^p\overline{\mathrm{H}_{\mathrm{dR}}^l}(R) \neq \{0\}$ for $p >\frac{h}{w_l}$; indeed $\frac{h}{w_l}$ and $\frac{h}{W_{n-1}}$ are H\"older conjugated (since $w_l + W_{n-1} = h$).
\end{proof}

\begin{proof}[Proof of Proposition \ref{non-vanishing-proposition}]
Let $\phi^t _j:= R_{\exp tY_j}$ be the flow of $Y_j$, and set $\omega:= f \mathrm{vol}$ so that $\psi = \pi^* (\omega)$. 
Then for every $j \in J$ the form $\theta _j:= \int _{-\infty}^0 (\phi^t _j)^* (\iota _{Y_j} \omega)~ dt$ is a primitive of $\omega$. Indeed 
by applying the Cartan formula $\mathcal L _{Y_j} = d \circ \iota _{Y_j} + \iota _{Y_j} \circ d$ to the compactly suported closed form $\omega$, one gets
$$d\theta _j = \int _{-\infty}^0 (\phi^t _j)^* (d\iota _{Y_j} \omega)~ dt =
\int _{-\infty}^0 (\phi^t _j)^* (\mathcal L _{Y_j} \omega)~ dt =
\int _{-\infty}^0 \frac{d}{dt}\bigl( (\phi^t _j)^* \omega \bigr)~ dt,$$
which is equal to $\omega$. 

For $X \in \frak b$, we want to estimate the $L^p$ norms of $C^* _{\exp X} (\omega)$ and of its primitives $C^* _{\exp X} (\theta _j)$,
in order to apply Lemma \ref{non-vanishing-lemma} to the form $C^* _{\exp X} (\psi) = \pi^* (C^* _{\exp X} (\omega))$.

Observe that 
$$\bigl\Vert C^* _{\exp X} (\omega) \bigr\Vert _{L^p \Omega^n (N)} = \bigl\Vert (f \circ C _{\exp X}) \cdot C^* _{\exp X}(\mathrm{vol}) \bigr\Vert _{L^p \Omega^n (N )}
= \Vert \omega \Vert _{L^p \Omega^n (N)},$$ 
since $C_{\exp X}$ preserves the Riemannian volume on $N$ 
(by definition of the factor $\frak b <\frak a$). 

To estimate the $L^p$-norm of $C^* _{\exp X} (\theta _j)$, we proceed as follows. 
Define $F_j = \int _{-\infty}^0 f\circ \phi_j^t ~dt$. By assumption (3) the function $F_j$ has compact support. Since $\mathrm{vol}$ is bi-invariant on $N$,
one has $\theta _j = F_j \cdot (\iota _{Y_j} \mathrm{vol})$; and thus, since $\mathrm{vol}$ is $C _{\exp X}$-invariant:
$$C^* _{\exp X} (\theta _j) = (F_j \circ C_{\exp X}) \cdot (\iota _{C_{\exp X}^*(Y_j)} \mathrm{vol}).$$
By assumption (2), one has $\mathrm{Ad}(\exp X)Z_j = e^{-\varpi_j(X)} Z_j$. Since $Y_j = \Re Z_j$, and since it is left-invariant, one has 
$\vert C_{\exp X}^*(Y_j)\vert_g = e^{-\Re \varpi_j(X)} \vert Y_j \vert$, for every $g \in N$.
Thus for every $p \in (1, +\infty)$:
$$\bigl\Vert C^* _{\exp X} (\theta _j) \bigr\Vert _{L^p \Omega^{n-1} (N)} = e^{-\Re \varpi_j(X)} \vert Y_j \vert \Vert F_j \Vert_{L^p(N)}.$$
Note that all the $Y_j$'s are non-zero, indeed the operators $\mathrm{ad}X$ are \emph{real} endomorphisms.

Now assumption (1), in combination with the last equality and compactness of the unit sphere in $\frak b$, implies that there exist 
constants $c, C >0$ such that for every $p \in (1, +\infty)$ and 
every $X \in \frak b$, one has 
$$\inf _{j \in J} \bigl\Vert C^* _{\exp X} (\theta _j) \bigr\Vert _{L^p \Omega^{n-1} (N)} \leqslant Ce^{-c \vert X \vert}.$$ 
Thus, by Lemma \ref{non-vanishing-lemma}, 
the forms $\psi$ and $C^* _{\exp X} (\psi)$ belong to $\mathcal Z^{p, n}(H, \xi)$ for $1<p<\frac{h}{W{n-1}}$. Moreover the norm of $C^* _{\exp X} (\psi)$ decreases exponentially fast to $0$ when $\vert X \vert \to +\infty$.
The statement follows.
\end{proof}

\section{First applications: proof of two non-vanishing results}
\label{s - proof of non-vanishings}

This section is dedicated to the proof of the first two non-vanishing theorems, namely Theorem \ref{introduction-theorem2} and Theorem \ref{introduction-theorem3} mentioned in the introduction; they deal with general solvable groups satisfying conditions (nC) and (nT). 

\subsection{Functions with prescribed vanishing integrals}
\label{ss - suitable functions} 
The following lemmas will serve to exhibit functions satisfying condition (3) of Proposition \ref{non-vanishing-proposition}; these functions will eventually provide suitable non-zero forms leading to our targeted non-vanishing results by Corollary \ref{non-vanishing-corollary}. 

\begin{lemma}\label{criterium-lemma2}
Let $Y, T$ be non-trivial left-invariant vector fields on a connected simply connected nilpotent Lie group $N$. 
Then $N$ admits a non-zero smooth compactly supported function, whose integral along every orbit of $Y$ and $T$ is null.
\end{lemma}

\begin{proof} Let $\frak n$ be the Lie abgebra of $N$, and let $\frak l$ be the subalgebra generated by $Y$ and $T$.
We will prove the lemma by induction on the length of the descending central series of $\frak l$, \emph{i.e.\!} on the smallest
integer $k \geqslant 1$ such that $\frak l^{k+1} =0$ (where $\frak l^1:= \frak l$, $\frak l^2:= [\frak l, \frak l]$, $\frak l^{i+1}:= [\frak l, \frak l^i]$).
Let denote the length by $\mathrm{length}(\frak l)$.

If $\mathrm{length}(\frak l) =1$, then $Y$ and $T$ commute, and the statement follows from (the proof of) Lemma \ref{criterium-lemma1}.
Suppose now that $\mathrm{length}(\frak l) >1$.

Let $\phi_Y^t$ and $\phi _T^t$ be the flows of $Y$ and $T$. 
Let $f_0$ be a non-zero compactly supported smooth function on $N$. 
The flow of every non-trivial left-invariant vector field acts properly on such a Lie group $N$; 
thus there exists $t_0 \in{\bf R}$ with 
$$(\mathrm{support}f_0) \cap \phi_T^{-t_0}(\mathrm{support}f_0) = \varnothing.$$ 
The function $f:= f_0 - f_0 \circ \phi_T^{t_0}$ is non-zero, smooth,
compactly supported, and its integrals along the $T$-orbits are null. 

We are looking for a sufficient condition on $f_0$ which garanties 
that the integrals of $f$ along the $Y$-orbits are also null. 
For $g \in N$ one has 
$$\int _{\bf R} (f \circ \phi _Y^t)(g) ~dt = \int _{\bf R} (f_0 \circ \phi _Y^t)(g) ~dt - 
\int _{\bf R} (f_0 \circ \phi _T^{t_0} \circ \phi _Y^t \circ \phi _T^{- t_0})(\phi _T^{t_0}(g)) ~dt.$$
Thus, for the integrals of $f$ along the $Y$-orbits to be null, 
it is enough that the integrals of $f_0$ along the orbits of 
$Y$ and of $(\phi _T^{-t_0})^*(Y)$ are null.

The vector field $(\phi _T^{-t_0})^*(Y)$ is left-invariant. Its value at $1_N$ is 
$$Z:= \mathrm{Ad}(-t_0 T) Y = e^{\mathrm{ad}(-t_0 T)} Y.$$
According to the induction hypothesis, such an $f_0$ exists if the subalgebra $\frak m$ generated by $Y$ and $Z$
satisfies $\mathrm{length}(\frak m) < \mathrm{length}(\frak l)$.

Since $Y, Z \in \frak l$, one has $\frak m \subset \frak l$. Moreover:
$$[Y,Z] = [Y, e^{\mathrm{ad}(-t_0 T)}Y] = [Y, Y] - t_0 \bigl[Y, [T, Y]\bigr] + \dots \in \frak l^3.$$
It follows that $\frak m^2 \subset \frak l^3$, thus 
$\frak m^3 = [\frak m, \frak m^2] \subset [\frak l, \frak l^3] = \frak l^4$, \dots,
and so $\frak m^i \subset \frak l^{i+1}$ for every $i \geqslant 2$.
Therefore one has $\mathrm{length}(\frak m) < \mathrm{length}(\frak l)$ as expected.
\end{proof}

\begin{lemma}
\label{criterium-lemma1} 
Let $M$ be a smooth manifold and $Y_1,\dots, Y_k, T$ be a family of complete smooth vector fields. Assume that $T$
commutes with all the $Y_i$'s, and that its flow acts properly on $M$.
Suppose that the subfamily $Y_1,\dots, Y_k$ satisfies the following property: 
there exists a non-zero $C^1$ compactly supported function on $M$, 
whose integral along every orbit of the $Y_i$'s is null. Then the same property holds for the entire family $Y_1,\dots, Y_k, T$.
\end{lemma}

\begin{proof} Let $f_0$ be a function satisfying the property for the subfamily $Y_1, \dots, Y_k$. Let $\phi^t$ be the flow of $T$.
Since it acts properly on $M$ there exists $t_0 \in{\bf R}$ such that 
$(\mathrm{support}f_0) \cap \phi^{-t_0}(\mathrm{support}f_0) = \varnothing$.
Define $f = f_0 - f_0 \circ \phi^{t_0}$. It is a non-zero $C^1$ compactly supported function whose integral along every orbit of $T$ is null. 
Since $T$ commutes with the $Y_i$'s, the integrals of $f$ along their orbits remain null.
\end{proof}

\subsection{Two non-vanishing results for solvable groups}
We can now give the proofs of Theorems \ref{introduction-theorem2} and \ref{introduction-theorem3}. 

\begin{proof}[Proof of Theorem
\ref{introduction-theorem2}]
We can assume that $N = {\bf R}^n$.
With the notations of Subsection \ref{ss - criterium}, let $J \subset \{1, \dots, n\}$ be a minimal subset such that 
$$\{\varpi _j: j \in J\} \cup \{\overline{\varpi _j}: j \in J \} = \{\varpi_1, \dots, \varpi_n\}$$ 
as sets of linear forms on $\frak a$. It satisfies assumption
(1) of Proposition \ref{non-vanishing-proposition}.
For every $j \in J$, let $Z_j$ be as in assumption (2) of Proposition \ref{non-vanishing-proposition}.
The orbits in ${\bf R}^n$ of the vector field $Y_j:= \Re Z_j$ are the lines parallel to $Y_j$.
They commute. Thus Lemma \ref{criterium-lemma1} provides a function which satisfies assumption (3) of 
Proposition \ref{non-vanishing-proposition}.
Now Theorem \ref{introduction-theorem2} follows from Corollary \ref{non-vanishing-corollary}.
\end{proof}

\begin{proof}[Proof of Theorem 
\ref{introduction-theorem3}]
Since $A \simeq {\bf R}^2$, one has $B \simeq \mathbf R$. Let $X_0$ be a non-zero vector in $\frak b$. 
Let $J = \{j_1, j_2\} \subset \{1, \dots, n\}$
be such that $\Re (\varpi _{j_1} (X_0)) >0$ and $\Re (\varpi _{j_2} (-X_0)) >0$. It satisfies assumption
(1) of Proposition \ref{non-vanishing-proposition}. 
For every $j \in J$, let $Z_j$ be as in assumption (2) of Proposition \ref{non-vanishing-proposition}, and let 
$Y_j:= \Re Z_j$. By applying Lemma \ref{criterium-lemma2} to the pair of left-invariant vector fields $Y_{j_1}, Y_{j_2}$,
one obtains a function that satisfies assumption (3) of 
Proposition \ref{non-vanishing-proposition}.
Now Theorem \ref{introduction-theorem3} follows from Corollary \ref{non-vanishing-corollary}.
\end{proof}

\section{Semisimple Lie groups}
\label{s - semisimple} 

In this section, we are looking for connected solvable Lie subgroups $R$ in semisimple real Lie groups which are of the form $R = A \ltimes N$ with $A \simeq {\bf R}^l$ and $l \geqslant 1$. 
Of course these groups will satisfy the contraction condition (nC) and the non-triviality condition (nT). 
More precisely, our goal is to prove that in any semisimple real Lie group with finite center, the solvable subgroups appearing in Iwasawa decompositions fulfill 
the assumptions of our non-vanishing criterion Corollary \ref{non-vanishing-corollary}. As a consequence we obtain a proof of Theorem \ref{introduction-theorem4},
hence a proof for Theorem \ref{introduction-theorem1}, in view of the reduction contained in \ref{ss - ss and QI}. 
From a technical viewpoint, the proofs in this section are ultimately relevant to the combinatorics of root systems.

\subsection{Lie-theoretic notions and notations}
\label{ss - combinatorial problem} 
Let $G$ be a semisimple real Lie group with finite center. 
We pick in $G$ a subgroup $A$ which is maximal for the properties of being connected, abelian and diagonalizable over the real numbers in the adjoint representation of $G$. 
We denote the latter representation by 
\[
{\rm Ad}: G \to {\rm GL}({\frak g}), 
\]
with ${\frak g} = {\rm Lie}\, G$ .
We have a direct sum decomposition 
\[
{\frak g} = {\frak g}_0 \oplus \bigoplus_{\alpha \in \Phi(G,A)} {\frak n}_\alpha,
\]
where ${\frak g}_0$ is the subspace on which the adjoint $A$-action is trivial and where ${\frak n}_\alpha$ is the weight space associated with the character $\alpha$ of $A$: 
\[
{\frak n}_\alpha = \{X \in \frak{g}: {\rm Ad}(a)X = \alpha(a)X \,\, \hbox{\rm for all} \,\, a \in A\}.
\]

We also pick a minimal parabolic subgroup $P$ containing $A$, which provides a basis and a positive root subset $\Phi^+$ for the root system $\Phi=\Phi(G,A)$ of $G$. 
In order to fit with the notation of \cite{BBK-Lie-4to6}, we set $V=\frak a^*$. 
This real vector space is equipped with the Killing form, which makes it a Euclidean space; this enables us to identify $\frak a$ and $\frak a^*$. 
Finally, we set: 
\[
2 \rho_G = \sum_{\alpha \in \Phi^+} {\rm dim}({\frak n}_\alpha) \alpha, 
\]
and we denote by $N$ the nilpotent group integrating $\bigoplus_{\alpha \in \Phi^+} {\frak n}_\alpha$. 
Thus, $R=AN$ is the solvable part of an Iwasawa decompostion of $G$, and all such solvable subgroups of $G$ can be obtained by varying the choices of $A$ and $P$ above. 
The remaining choice is that of a maximal compact subgroup $K$: the stabilizer of any point in the maximal flat associated to $A$ in the Riemmanian symmetric space associated to $G$ does the job. 

From a combinatorial viewpoint, in the main part of this section we are looking for a subset $\Psi \subset \Phi^+$ and a family of real numbers $(m_\beta)_{\beta \in \Psi}$, with $m_\beta > 0$ for each $\beta \in \Psi$, such that 

\smallskip 

\begin{enumerate}
\item[{\rm (i)}]~we have: $2\rho_G = \sum_{\beta \in \Psi} m_\beta \beta$, 
\item[{\rm (ii)}]~we have: $V = \sum_{\beta \in \Psi} {\bf R} \beta$, 
\item[{\rm (iii)}]~if $\beta, \beta' \in \Psi$ then: $\beta+\beta' \not\in \Phi$. 
\end{enumerate} 

\noindent
We will sometimes relax condition (iii) to the following one 

\vspace{3mm} 
\begin{enumerate}
\item[${\rm (iii)}'$]~if $\beta, \beta' \in \Psi$ then: $\beta+\beta' \not\in \Phi$ except maybe for one pair $\{ \beta;\beta' \} \subset \Psi$. 
\end{enumerate} 

\subsection{Relationship with the non-vanishing criterion} 
Let us explain why the above combinatorial conditions are relevant to the non-vanishing criterion of the previous section. 
In \ref{ss - criterium}, the set $\{1, 2,\dots, n \}$ is determined by the dimension of the nilpotent group $N$: in the present section, we see it more concretely as the set of positive roots counted with their multiplicities; accordingly, the subset $J$ of \ref{ss - criterium} is denoted here by $\Psi$ in view of its interpretation in terms of roots. 
Note also that, by the choice of $A$ here (a maximal ${\bf R}$-split torus), we do not need to extend the scalars to ${\bf C}$ in our context. 

First of all, the solvable group $R=AN$ from \ref{ss - combinatorial problem} satisfies condition (nC) since we can choose $-\xi$ to be the sum of the fundamental coweights for the root system of $G$ (in fact, we can take any vector in $\frak{a}^+$, \emph{i.e.}~in the Weyl cone given by $AN$, tangent to a regular geodesic in the symmetric space). 
The group $R$ also satisfies condition (nT) by definition of a root as a \emph{non-trivial} character of $A$ (or $\frak{a}$). 
We now concentrate on checking the conditions of the non-vanishing criterion in the present situation. 

\begin{lemma}
\label{lemma - ss OK} 
We assume we are given $\Psi$ satisfying {\rm (i)-(iii)}$'$. 
Then the choice $J = \Psi$ fulfills the assumptions of Proposition \ref{non-vanishing-proposition}. 
\end{lemma} 

The lemma will lead to the desired non-vanishing statement, hence most of the rest of the Section will be dedicated to exhibiting such a subset $\Psi$ for any semisimple 
real Lie group $G$.

\begin{proof}
As mentioned in Remark \ref{rk - non-vanishing condition}, Condition (2) is automatically satisfied by definition of the adjoint representation and of the roots. 

Let us check Condition (1). 
By (i), we have: 
\[ 
\sum_{\alpha \in \Phi^+} {\rm dim}({\frak n}_\alpha) \alpha = 2 \rho_G = \sum_{\beta \in \Psi} m_\beta \beta.
\]
Now let $X \in \frak{b}$. 
Note first that for any $X \in \frak{a}$, we have: $\mathrm{trace}(\mathrm{ad}X) = 2 \rho_G(X)$; therefore $\frak{b}$, defined in \ref{introduction-strategy} by 
$\frak b=\{X \in \frak a: \mathrm{trace}(\mathrm{ad}X) = 0\}$, can be seen as $\frak b=\{X \in \frak a: \rho_G(X) = 0\}$. 
We assume that $\beta(X) \leqslant 0$ for each $\beta \in \Psi$. 
Since $m_\beta > 0$ for each $\beta \in \Psi$, we deduce that $\beta(X)=0$ for each $\beta \in \Psi$. 
Finally this implies that $X=0$ by (ii), proving that (1) is satisfied. 

Condition (3) is more delicate to check but it is treated thanks to the lemmas in \ref{ss - suitable functions} combined with the fact that if for $\alpha, \alpha' \in \Phi^+$ we have $\alpha + \alpha' \not\in \Phi$ then $[\frak{n}_\alpha,\frak{n}_{\alpha'}]=\{ 0 \}$. 
Indeed, in view of this, the function requested by Condition (3) is directly given by an inductive use of Lemma \ref{criterium-lemma1} in case $\Psi$ satisfies (iii); if $\Psi$ satisfies (iii)$'$ only, then one has to apply first Lemma \ref{criterium-lemma2}Â   to the pair of roots in $\Psi$ whose sum is a root, and then again apply inductively Lemma \ref{criterium-lemma1}. 
\end{proof}

In the rest of the Section, we are thus reduced to exhibiting subsets $\Psi$ satisfying (i)-(iii)$'$, or even better satisfying (i)-(iii). 
Note that in view of Theorem \ref{introduction-theorem3}  we may -- and shall -- assume that the rank $l = {\rm dim}(A)$ of $G$ is at least $3$.

\subsection{A preliminary result in linear algebra}
\label{ss - linear algebra}
We use the notation from \cite[Planches]{BBK-Lie-4to6}, and more precisely the concrete descriptions of root systems in terms of linear algebra.
In particular, we consider the standard Euclidean space ${\bf R}^l$ with canonical basis $(\varepsilon_i)_{1 \leqslant i \leqslant l}$. 
We are interested in sums of the form 
\begin{align*}
\sum_{m=1}^M m(\varepsilon_{l-2m+1} + \varepsilon_{l-2m} + & ~\varepsilon_{l-2m} + \varepsilon_{l-2m-1})=\\ 
 &= \sum_{m=1}^M m(\varepsilon_{l-2m+1} + 2\varepsilon_{l-2m} + \varepsilon_{l-2m-1}), 
\end{align*}
where $M$ is a suitably chosen integer. 
We will see below that such linear combinations are very useful to approximate the sum of positive roots of some fixed norm in suitable root systems. 

In addition, the vectors appearing in these sums are sums of two consecutive vectors of the form $\varepsilon_{i}+\varepsilon_{i+1}$: in many root systems, such a vector $\varepsilon_{i}+\varepsilon_{i+1}$ is a positive root and the sum of two such vectors is {\it not}~a root, which is useful to achieve condition (iii), or maybe (iii)$'$, above. 

More precisely, in many root systems the sum of positive roots of norm equal to $\sqrt{2}$ is the vector $\sigma$ given by: 
\[
\sigma = 
\sum_{i<j} \varepsilon_i \pm \varepsilon_j 
= \sum_{i=1}^{l-1} \bigl( \sum_{j>i} (\varepsilon_i + \varepsilon_j + \varepsilon_i - \varepsilon_j) \bigr) 
= 2 \sum_{i=1}^{l-1} (l-i) \varepsilon_i
= 2 \sum_{i=1}^{l-1} i \varepsilon_{l-i}, 
\]
and we want to see $\sigma$ as a linear combination, with positive coefficients, of vectors $\varepsilon_{i}+\varepsilon_{i+1}$. 

\begin{lemma}
\label{lemma - linear algebra}
Besides the vector $\displaystyle \sigma = \sum_{1 \leqslant i<j \leqslant l} \varepsilon_i \pm \varepsilon_j = 2 \sum_{i=1}^{l-1} i \varepsilon_{l-i}$ above, we introduce the sum 
$\displaystyle 
S = \sum_{m=1}^{\lfloor {l-2 \over 2} \rfloor} m(\varepsilon_{l-2m+1} + \varepsilon_{l-2m} + \varepsilon_{l-2m} + \varepsilon_{l-2m-1})$. 
Then: 
\begin{itemize} 
\item[$\bullet$] if $l$ is even, we have: $\sigma = 2S + l \varepsilon_1$; 
\item[$\bullet$] if $l$ is odd, we have: $\sigma = 2S + (l-1)  (\varepsilon_1 + \varepsilon_2) + (l-1) \varepsilon_1$. 
\end{itemize} 
\end{lemma}

\begin{proof}
In the sum $S_M = \sum_{m=1}^M m(\varepsilon_{l-2m+1} + \varepsilon_{l-2m} + \varepsilon_{l-2m} + \varepsilon_{l-2m-1})$, we consider the sum of two consecutive terms, that is: 
\[
k(\varepsilon_{l-2k+1} + 2\varepsilon_{l-2k} + \varepsilon_{l-2k-1}) + (k+1)(\varepsilon_{l-2k-2+1} + 2\varepsilon_{l-2k-2} + \varepsilon_{l-2k-2-1})
\]
\[
= k \varepsilon_{l-2k+1} + 2k \varepsilon_{l-2k} + (2k+1) \varepsilon_{l-2k-1} + (2k+2) k \varepsilon_{l-2k-2} + (k+1) \varepsilon_{l-2k-3}.
\]

The three middle terms also appear in ${1 \over 2} \sigma$. 
The maximal index of summation $M$ is the biggest integer $m$ satisfying $l-2m-1 \geqslant 1$, that is $m \leqslant {l-2 \over 2}$. 
This explains why we are led to considering: 
\[ 
S = S_{\lfloor {l-2 \over 2} \rfloor} = \sum_{m=1}^{\lfloor {l-2 \over 2} \rfloor} m(\varepsilon_{l-2m+1} + \varepsilon_{l-2m} + \varepsilon_{l-2m} + \varepsilon_{l-2m-1}). 
\]
The sum $S = S_{\lfloor {l-2 \over 2} \rfloor}$ approximates ${1 \over 2} \sigma$ down to the index $l -2 \lfloor {l-2 \over 2} \rfloor - 1$. 

$\bullet$ If $l$ is even, we have $\lfloor {l-2 \over 2} \rfloor = {l-2 \over 2}$, then $l -2 \lfloor {l-2 \over 2} \rfloor - 1 = 1$ and therefore: 
\begin{align*}
S = &\sum_{m=1}^{\lfloor {l-2 \over 2} \rfloor} m(\varepsilon_{l-2m+1} + \varepsilon_{l-2m} + \varepsilon_{l-2m} + \varepsilon_{l-2m-1})\\
= &\sum_{j = 1}^{l-2} j \varepsilon_{l-j} + ({l \over 2} - 1) \varepsilon_1 = {1 \over 2} \sigma - {l \over 2} \varepsilon_1, 
\end{align*}
so that in this case: 
\[
{1 \over 2} \sigma = S + {l \over 2} \varepsilon_1.
\]

$\bullet$ If $l$ is odd, we have $\lfloor {l-2 \over 2} \rfloor < {l-2 \over 2}$, then $l -2 \lfloor {l-2 \over 2} \rfloor - 1 = l - (l-3) - 1 = 2$ and therefore: 
\[
S = \sum_{m=1}^{\lfloor {l-2 \over 2} \rfloor} m(\varepsilon_{l-2m+1} + \varepsilon_{l-2m} + \varepsilon_{l-2m} + \varepsilon_{l-2m-1})
= \sum_{j = 1}^{l-3} j \varepsilon_{l-j} + \lfloor {l-2 \over 2} \rfloor \varepsilon_2. 
\]
If we denote $l=2r+1$ with $r \geqslant 1$, then $\lfloor {l-2 \over 2} \rfloor = \lfloor {2r-1 \over 2} \rfloor = \lfloor r-{1 \over 2} \rfloor = r-1$ and we have: 
\[
S = \sum_{j = 1}^{l-3} j \varepsilon_{l-j} +(r-1) \varepsilon_2,
\]
so that, since the coefficient of $\varepsilon_2$ in ${1 \over 2} \sigma$ is $l-2 = (2r+1) - 2 = 2r-1$, we have 
${1 \over 2} \sigma = S + r(\varepsilon_1 + \varepsilon_2) + r \varepsilon_1$ and finally
\[
{1 \over 2} \sigma = S + {l-1 \over 2}  (\varepsilon_1 + \varepsilon_2) + {l-1 \over 2} \varepsilon_1.
\]
This concludes the proof. 
\end{proof}

At this stage, in order to find subsets $\Psi$ of positive roots achieving the conditions (i)-(iii)$'$ at the end of Section \ref{ss - combinatorial problem}, we use the classification of simple real Lie groups. 
The parameters of this classification are: a (possibly non-reduced) root system, and the multiplicities of the roots. 
According to \'E.~Cartan's classification, a simple real Lie group either is a simple complex Lie group seen as a real one (in which case the root system is reduced and all multiplicities are equal to 2), or is absolutely simple (\emph{i.e.}~ its Lie algebra stays simple after complexification) and belongs to the list given for instance in \cite[Chapter X, Table VI, pp.\! 532-534]{Helgason}. 
One useful fact is that the Weyl group acts transitively on roots of given norm (see \cite[Lemma C, p.\! 53]{Humphreys-LA}), 
so that not so many possibilities of root multiplicities appear (one multiplicity if the root system is simply laced \emph{i.e.}~if the edges in the Dynkin diagram are all simple, 
at most 2 if it is reduced, at most 3 otherwise). 
This also implies that $2 \rho_G = \sum_{\alpha \in \Phi^+} {\rm dim}({\frak n}_\alpha) \alpha$ can be computed as a sum of at most 3 partial sums, namely packets of roots of given norm multiplied by the corresponding multiplicity. 

In the rest of this Section, we check that we can find a suitable root subset $\Psi$ for each isomorphism class of simple real Lie groups by sorting them according to their relative root system first and then their multiplicities. 
The reader will easily check that all Cartan types appearing in \cite[Chapter X, Table VI, pp.\! 532-534]{Helgason} are covered.

\subsection{All types except $\mathrm A_l$}
To be consistent with the notation of \cite[Chapter X, Table VI, pp.\! 532-534]{Helgason}, we denote by $l$ the real rank of $G$ and by $r$ its complex rank (recall that the latter rank is the dimension of a subalgebra in $\frak{g} \otimes_{\bf R} {\bf C}$ consisting of diagonalizable elements in the adjoint representation, and maximal for this elementwise property). 
Recall that we can restrict our attention to the case $l \geqslant 3$, thanks to Theorem \ref{introduction-theorem3}. 

$\bullet$ {\bf Type ${\rm B}_l$} 
\cite[Planche II, p.\! 252]{BBK-Lie-4to6}. 
This is a non simply laced root system, and indeed some cases with different multiplicities do appear in Cartan's classification. 
This is the root system in the standard Euclidean space ${\bf R}^l$ with canonical basis $(\varepsilon_i)_{1 \leqslant i \leqslant l}$, where the positive roots are the vectors of the form $\varepsilon_i$ (there are $l$ such roots, of norm 1) or $\varepsilon_i \pm \varepsilon_j$ with $i<j$ (there are $l(l-1)$ such roots, of norm $\sqrt{2}$). 
We still denote the partial sum of roots of norm $\sqrt{2}$ by 
\[
\sigma = \sum_{i<j} \varepsilon_i \pm \varepsilon_j = 2 \sum_{i=1}^{l-1} i \varepsilon_{l-i}, 
\]
and we denote the partial sum of roots of norm $1$ by $\tau = \sum_{i=1}^l \varepsilon_i$.
The Cartan types where this root system appears are the types ${\rm B}_{\rm I}$ 
and ${\rm D}_{\rm I}$, 
in which cases we have: 
\[ 
2 \rho_{{\rm B}_{\rm I}} = \sigma + \bigl( 2(r-l) + 1 \bigr) \tau 
\quad \hbox{\rm and} \quad 
2 \rho_{{\rm D}_{\rm I}} = \sigma + 2(r-l) \tau. 
\]
Let us treat thoroughly the Cartan type ${\rm B}_{\rm I}$, the type ${\rm D}_{\rm I}$ being similar. 
We invoke Lemma \ref{lemma - linear algebra}, together with the associated notation. 

If $l$ is even, we have: 
\[
2 \rho_{{\rm B}_{\rm I}} = 2 S + l \varepsilon_1 + \bigl( 2(r-l) + 1 \bigr) \sum_{m=1}^{{l \over 2}} (\varepsilon_{2m-1} +\varepsilon_{2m})
\]
so that we can take the subset of positive roots $\Psi$ to be consisting of $\varepsilon_1$ and of the roots $\varepsilon_{j} + \varepsilon_{j+1}$ for $1 \leqslant j \leqslant l-1$. 
This choice of $\Psi$ gives immediately (i) in view of the definition of the sum $S$ and (ii), \emph{i.e.}~ the fact that $\Psi$ generates $V$, is easy; at last (iii) follows from the fact that no coefficient 2 appears in the coordinates of roots and the fact that the support of roots has cardinality $\leqslant 2$. 

If $l$ is odd, we have 
\begin{align*} 
2 \rho_{{\rm B}_{\rm I}} = & ~2 S + 2 (l-1)(\varepsilon_1+\varepsilon_2) + (l-1)\varepsilon_1 \\ 
+ & ~\bigl( 2(r-l) + 1 \bigr) \sum_{m=1}^{l-1 \over 2} (\varepsilon_{2m} +\varepsilon_{2m+1}) + \bigl( 2(r-l) + 1 \bigr) \varepsilon_1. 
\end{align*}
The expression is more complicated but it gives (i) and the same arguments as above work to give (ii) and (iii). 

At last, as already mentioned, the Cartan type ${\rm D}_{\rm I}$ is deduced from this case after replacing the coefficient $ \bigl( 2(r-l) + 1 \bigr)$ by $2(r-l)$.

$\bullet$ {\bf Type ${\rm C}_l$} 
\cite[Planche III, p.\! 254]{BBK-Lie-4to6}. 
This is a non simply laced root system, and again some cases with different multiplicities do appear in Cartan's classification. 
This is the root system in the standard Euclidean space ${\bf R}^l$ with canonical basis $(\varepsilon_i)_{1 \leqslant i \leqslant l}$, where the positive roots are the vectors of the form $2 \varepsilon_i$ (there are $l$ such roots, of norm 2) or $\varepsilon_i \pm \varepsilon_j$ with $i<j$ (there are $l(l-1)$ such roots, of norm $\sqrt{2}$). 
We still denote the partial sum of roots of norm $\sqrt{2}$ by 
\[
\sigma = \sum_{i<j} \varepsilon_i \pm \varepsilon_j = 2 \sum_{i=1}^{l-1} i \varepsilon_{l-i}, 
\]
and we denote the partial sum of roots of norm $2$ by $\kappa = \sum_{i=1}^l 2\varepsilon_i$.
The Cartan types covered by this root system are: some cases ${\rm A}_{\rm III}$, all cases ${\rm C}_{\rm I}$, some cases ${\rm C}_{\rm II}$, half of the cases ${\rm D}_{\rm III}$ and the case ${\rm E}_{\rm VII}$. 
For any such type, say ${\rm G}$, we have: 
\[ 
2 \rho_{\rm G} = m_\sigma \sigma + m_\tau \kappa
\]
where $m_\sigma$ and $m_\tau$ are integers $\geqslant 1$\footnote{More precisely, we have: 
$2 \rho_{{\rm A}_{\rm III}} = 2 \sigma + \kappa$ in the ${\rm C}_l$ case of ${\rm A}_{\rm III}$, 
$2 \rho_{{\rm C}_{\rm I}} = \sigma + \kappa$, 
$2 \rho_{{\rm C}_{\rm II}} = 4 \sigma + 3 \kappa$ in the ${\rm C}_l$ case of ${\rm C}_{\rm II}$, 
we have $2 \rho_{{\rm D}_{\rm III}} = 4 \sigma + \kappa$ in the ${\rm C}_l$ case of ${\rm D}_{\rm III}$ 
and 
$2 \rho_{{\rm E}_{\rm VII}} = 8 \sigma + \kappa$, the root system being ${\rm C}_3$ in the latter case.}. 

We choose for $\Psi$ the set of roots $2 \varepsilon_j$ for $1 \leqslant j \leqslant l$. 
Conditions (i) to (iii) are trivially satisfied, for a choice of integral coefficients in (i) thanks to the coefficients 2 in the expressions of $\sigma$ and $\tau$.

$\bullet$ {\bf Type ${\rm BC}_l$}
\cite[Chapitre VI.14, p.\! 222]{BBK-Lie-4to6}. 
This is the only type of non-reduced root systems: the system ${\rm BC}_l$ consists of the union of the system ${\rm B}_l$ and of the system ${\rm C}_l$ as described above. 
The Cartan types covered by this root system are: the remaining cases ${\rm A}_{\rm III}$ and ${\rm C}_{\rm II}$, as well as the remaining half of the cases ${\rm D}_{\rm III}$. 
More precisely, still denoting by $r$ the complex rank, we have 
\[ 
2 \rho_{{\rm A}_{\rm III}} = 2(r-2l+1)\tau+ 2 \sigma+ \kappa
\]
in the remaining cases of ${\rm A}_{\rm III}$, 
\[ 
2 \rho_{{\rm C}_{\rm II}} = 4(r-2l) \tau+ 4 \sigma+ 3 \kappa 
\]
in the remaining cases of ${\rm C}_{\rm II}$, and 
\[ 
2 \rho_{{\rm D}_{\rm III}} = 4 \tau + 4 \sigma+ \kappa 
\]
in the remaining cases of ${\rm D}_{\rm III}$. 

This type is treated by taking $\Psi = \{2 \varepsilon_j \}_{1 \leqslant j \leqslant l}$, as for the previous case. 
Conditions (i) to (iii) are trivially satisfied, for a choice of integral coefficients in (i).

$\bullet$ {\bf Type ${\rm D}_l$}
\cite[Planche IV, p.\! 256]{BBK-Lie-4to6}. 
This is a simply laced type but in fact it occurs only once in Cartan's classification, namely for ${\rm D}_{\rm I}$.
This is the root system in the standard Euclidean space ${\bf R}^l$ with canonical basis $(\varepsilon_i)_{1 \leqslant i \leqslant l}$, where the positive roots are the vectors of the form $\varepsilon_i \pm \varepsilon_j$ with $i<j$ (there are $l(l-1)$ such roots, of norm $\sqrt{2}$). 
In this situation, we have 
\[
2 \rho_{{\rm D}_{\rm I}} = \sigma = 2 \sum_{i=1}^{l-1} i \varepsilon_{l-i}. 
\]
We use the subset of positive roots $\Psi$ to be consisting of the roots $\varepsilon_{j} + \varepsilon_{j+1}$ for $1 \leqslant j \leqslant l-1$, together with the roots $\varepsilon_1 \pm \varepsilon_l$. 
Condition (ii) is checked by the fact that $\varepsilon_1$ is the average of the last two roots, and then the other canonical vectors are obtained by an easy induction. 
Condition (iii)$'$ is checked by seeing that the only pair of roots in $\Psi$ whose sum is a root is $\{\varepsilon_1-\varepsilon_l , \varepsilon_{l-1}+\varepsilon_l \}$, again by considerations of support or coefficient $\geqslant 2$. 
For condition (i), we use Lemma \ref{lemma - linear algebra}. 
If $l$ is even we have 
\[
2 \rho_{{\rm D}_{\rm I}} = 2 S + l \varepsilon_1 = 2 S + {l \over 2} \bigl( (\varepsilon_1 + \varepsilon_l) + (\varepsilon_1 - \varepsilon_l) \bigr), 
\]
and if $l$ is odd we have 
\[
2 \rho_{{\rm D}_{\rm I}} = 2 S + (l-1)(\varepsilon_1 + \varepsilon_2) + {l-1 \over 2} \bigl( (\varepsilon_1 + \varepsilon_l) + (\varepsilon_1 - \varepsilon_l) \bigr). 
\]

The rest of the root systems consists of exceptional types. 
In the case of ${\rm E}_6$, ${\rm E}_7$ and ${\rm E}_8$ the root systems are simply laced and can be realized in ${\bf R}^8$. 
We merely mention the targeted vector $\rho$, the subset $\Psi$ and the linear combination achieving (i). 

$\bullet$ {\bf Type ${\rm E}_6$}
\cite[Planche V, p.\! 260]{BBK-Lie-4to6}. It occurs in Cartan's classification for ${\rm E}_{\rm I}$ only. The underlying vector space
$V$ is the $6$-dimensional subspace $x_6 = x_7 = -x_8$ in $\mathbf R ^8$ endowed with the canonical orthonormal basis $\varepsilon _1, \dots , \varepsilon _8$. The positive roots are $\pm \varepsilon _i + \varepsilon _j$, for $1\leqslant i < j \leqslant 5$, and
$$\frac{1}{2}\bigl( \varepsilon_8 - \varepsilon _7 - \varepsilon _6 + \sum _{i=1} ^5 (-1)^{\nu (i)} \varepsilon _i \bigr),
~~~\mathrm{~~with~~}~~~ \sum _{i=1}^5 \nu (i)~~\mathrm{even}.$$
The targeted vector is 
\[
2 \rho_{{\rm E}_{\rm I}} = 2 \varepsilon_2 + 4 \varepsilon_3 + 6 \varepsilon_4 + 8 \varepsilon_5 + 8 (\varepsilon_8 - \varepsilon_7 - \varepsilon_6). 
\]
To simplify the notation we let $v := \varepsilon_8 - \varepsilon _7 - \varepsilon _6$, and we consider the positive roots
$\beta_1, \dots, \beta_5$ defined by:
\begin{itemize}
\item[] $2\beta_1 = v + \varepsilon_1 + \varepsilon_2 + \varepsilon_3 + \varepsilon_4 + \varepsilon_5$,
\item[] $2\beta_2 = v - \varepsilon_1 - \varepsilon_2 + \varepsilon_3 + \varepsilon_4 + \varepsilon_5$,
\item[] $2\beta_3 = v - \varepsilon_1 + \varepsilon_2 + \varepsilon_3 + \varepsilon_4 - \varepsilon_5$,
\item[] $2\beta_4 = v + \varepsilon_1 - \varepsilon_2 - \varepsilon_3 + \varepsilon_4 + \varepsilon_5$,
\item[] $2\beta_5 = v - \varepsilon_1 - \varepsilon_2 - \varepsilon_3 - \varepsilon_4 + \varepsilon_5$.
\end{itemize}
Then one checks easily that 
$$2 \rho_{{\rm E}_{\rm I}} = 6 \beta _1 + 2 \beta _2 + 4 \beta _3 + 2\beta _4 + 2 \beta_5 + 2(-\varepsilon _1+\varepsilon _5)
+2(\varepsilon _1 + \varepsilon _5).$$
We define $\Psi$ to be the union of $\beta_1, \dots, \beta_5$ with $\pm\varepsilon _1+\varepsilon _5$. Condition (i)
is then satisfied. To show that condition (ii) holds, one first observes that the two roots $\pm\varepsilon _1+\varepsilon _5$ 
generates $\varepsilon _1$ and $\varepsilon _5$. Moreover one has
\begin{itemize}
\item[] $\beta_2 -\beta_1 = -\varepsilon_1 - \varepsilon_2$,
\item[] $\beta_3 -\beta_1 = -\varepsilon_1 - \varepsilon_5$,
\item[] $\beta_4 -\beta_1 = - \varepsilon_2 - \varepsilon_3$,
\item[] $\beta_5 -\beta_1 = -\varepsilon_1 - \varepsilon_2 - \varepsilon_3 - \varepsilon_4$,
\end{itemize}
from which one obtains that $\Psi$ generates $\varepsilon_1, \dots, \varepsilon_5$. Condition (ii) follows easily.
Moreover $\{\beta _3, \varepsilon _1+\varepsilon _5\}$ is the only pair of roots in $\Psi$ whose sum is a root in ${\rm E}_6$. 
Thus condition (iii)$'$ is satisfied. 

$\bullet$ {\bf Type ${\rm E}_7$}
\cite[Planche VI, p.\! 266]{BBK-Lie-4to6}. 
The targeted vector is 
\[
2 \rho_{{\rm E}_{\rm V}} = 2 \varepsilon_2 + 4 \varepsilon_3 + 6 \varepsilon_4 + 8 \varepsilon_5 + 10 \varepsilon_6 - 17 \varepsilon_7 + 17 \varepsilon_8, 
\]
which we decompose as the sum of $2 \varepsilon_2 + 4 \varepsilon_3 + 6 \varepsilon_4 + 8 \varepsilon_5 + 10 \varepsilon_6$ and of $17(\varepsilon_8-\varepsilon_7)$. 
It turns out that $\varepsilon_8-\varepsilon_7$ is the longest root in the system; we take it in $\Psi$ together with the four roots $\varepsilon_2+\varepsilon_3, \varepsilon_3+\varepsilon_4, \varepsilon_4+\varepsilon_5, \varepsilon_5+\varepsilon_6$, as well as the roots $\varepsilon_6 \pm \varepsilon_1$. 
Then we have: 
\begin{align*}
2 \rho_{{\rm E}_{\rm V}} 
&= 2(\varepsilon_2+\varepsilon_3) + 2(\varepsilon_3+\varepsilon_4) + 4 (\varepsilon_4+\varepsilon_5) + 4(\varepsilon_5+\varepsilon_6)\\ &+ 3 (\varepsilon_6+\varepsilon_1) + 3(\varepsilon_6-\varepsilon_1) + 17(\varepsilon_8-\varepsilon_7).
\end{align*}
The subset $\Psi$ satisfies (i), (ii) and (iii), the latter condition being easily checked by the concrete description of ${\rm E}_7$ in terms of the canonical vectors $\varepsilon_i$.

$\bullet$ {\bf Type ${\rm E}_8$}
\cite[Planche VII, p.\! 268]{BBK-Lie-4to6}. 
The targeted vector is 
\[
2 \rho_{{\rm E}_{\rm VIII}} = 2 \varepsilon_2 + 4 \varepsilon_3 + 6 \varepsilon_4 + 8 \varepsilon_5 + 10 \varepsilon_6 + 12 \varepsilon_7 + 46 \varepsilon_8. 
\]
We use the subset $\Psi$ consisting of the roots $\varepsilon_2+\varepsilon_3$, $\varepsilon_3+\varepsilon_4$, $\varepsilon_4+\varepsilon_5$, $\varepsilon_5+\varepsilon_6$, $\varepsilon_6+\varepsilon_7$, $\varepsilon_7+\varepsilon_8$ together with 
$\varepsilon_8 \pm \varepsilon_1$. 
Then we have: 
\begin{align*}
2 \rho_{{\rm E}_{\rm V}} 
&= 2 (\varepsilon_2+\varepsilon_3) + 2 (\varepsilon_3+\varepsilon_4) + 4 (\varepsilon_4+\varepsilon_5) + 4 (\varepsilon_5+\varepsilon_6)\\ &+ 6 (\varepsilon_6+\varepsilon_7) + 6 (\varepsilon_7+\varepsilon_8) 
+ 20 (\varepsilon_8-\varepsilon_1) + 20 (\varepsilon_8+\varepsilon_1).
\end{align*}
The subset $\Psi$ satisfies (i), (ii) and (iii), the latter condition being easily checked by the concrete description of ${\rm E}_8$ in terms of the canonical vectors $\varepsilon_i$.

$\bullet$ {\bf Type ${\rm F}_4$}
\cite[Planche VIII, p.\! 272]{BBK-Lie-4to6}. 
This is a non simply laced root system corresponding to four cases in Cartan's classification, namely ${\rm E}_{\rm II}, {\rm E}_{\rm VII}, {\rm E}_{\rm IX}$ and ${\rm F}_{\rm I}$. 
The sum of the positive roots of norm $\sqrt{2}$ is $\sigma = 2 (3 \varepsilon_1 + 2 \varepsilon_2 + \varepsilon_3)$ and the sum of the positive roots of norm 1 is $\tau = 5 \varepsilon_1 + \varepsilon_2 + \varepsilon_3 + \varepsilon_4$. 
The targeted vectors are of the form $\sigma + 2^j \tau$ with $0 \leqslant j \leqslant 3$, so it is enough to treat separately $\sigma$ and $\tau$. 
We choose $\Psi = \{\varepsilon_1, \varepsilon_1+\varepsilon_2, \varepsilon_2+\varepsilon_3, \varepsilon_3+\varepsilon_4 \}$. 
Condition (ii) is clear and (iii) is checked by the concrete description of ${\rm F}_4$ in terms of the canonical vectors $\varepsilon_i$. 
For (i), we use the fact that: $\sigma = 2 (\varepsilon_1+\varepsilon_2) + 2 (\varepsilon_2+\varepsilon_3) + 4 \varepsilon_1$ and 
$\tau = (\varepsilon_1+\varepsilon_2) + (\varepsilon_3+\varepsilon_4) + 4 \varepsilon_1$.

$\bullet$ {\bf Type ${\rm G}_2$}
 \cite[Planche IX, p.\! 274]{BBK-Lie-4to6}. 
This is a root system of rank 2, corresponding to only one case in Cartan's classification, so we do not need to consider it. 
Still, we can simply say that using the root notation we have: 
 $2\rho_{\rm G} = 10 \alpha_1 + 6 \alpha_2$,
which can be written as: $2\rho_{\rm G} = \alpha_1 + 3(3\alpha_1+2\alpha_2)$.

\subsection{The type $\mathrm A_l$}
\label{ss - type A}
This type is more delicate because pairs of positive roots more often lead to a sum providing another root. 

According to \cite[Planche I, p.\! 250]{BBK-Lie-4to6}, type $\mathrm A_l$ is a simply laced root system in the subspace $V \subset \mathbf R^{l+1}$ 
defined by $\sum _{i=1} ^{l+1} x_i = 0$. The positive roots are the $\varepsilon_i - \varepsilon _j$, with $1\leqslant i<j \leqslant l+1$.

The Cartan types where it appears are the types ${\rm A}_{\rm I}$ and ${\rm A}_{\rm II}$.
One has 
$$2 \rho_{{\rm A}_{\rm I}} = l \varepsilon _1 + (l-2)\varepsilon _2 + (l-4) \varepsilon _3
+ \dots - (l-2) \varepsilon _l - l\varepsilon _{l+1},$$
and $\rho_{{\rm A}_{\rm II}} = 4 \rho_{{\rm A}_{\rm I}}$. The latter relation shows that the same root subset $\Psi$ holds for both type.
We will thus restrict ourself to the type ${\rm A}_{\rm I}$. The idea here is to write  $2 \rho_{{\rm A}_{\rm I}}$ by using 
long roots.

\begin{proposition}
$\mathrm{(1)}$ If $l$ is odd, write $l = 2k-1$ with $k \geqslant 1$.
Then: 
$$2 \rho _{{\rm A}_{\rm I}}= 2 \sum _{k<j-i} (\varepsilon _i - \varepsilon _j) 
+ \sum _{i = 1}^k (\varepsilon _i - \varepsilon _{i+k}).$$
Moreover if $\Psi$ is the set of roots that appear in the right side member, then $\Psi$ satisfies the conditions
$(i), (ii), (iii)$. 

$\mathrm{(2)}$ If $l$ is even, write $l =2k$ with $k \geqslant 1$.
Then:
$$2\rho _{{\rm A}_{\rm I}}= 
2\sum _{k<j-i<l} (\varepsilon _i - \varepsilon _j ) 
+ 2(\varepsilon _1 - \varepsilon _{k+1}) + 
2(\varepsilon _{k+1} - \varepsilon _{l+1}).$$
Moreover if $\Psi$ is the set of roots that appear in the right side member, then $\Psi$ satisfies the conditions
$(i), (ii), (iii)'$. 
\end{proposition}

\begin{proof} 
By counting the number of times $\varepsilon _i$ and $\varepsilon _j$
appear, one has:
$$\sum _{k<j-i} \varepsilon _i - \varepsilon _j = 
\sum _{i =1}^{l-k} (l+1-k-i)\varepsilon _i
- \sum _{j = k+2}^{l+1} (j-k-1)\varepsilon _j.$$

(1) Suppose $l = 2k-1$ with $k \geqslant 1$. Then:
$$\sum _{k<j-i} \varepsilon _i - \varepsilon _j = 
\sum _{i =1}^{k-1} (k-i)\varepsilon _i
- \sum _{j = k+2}^{l+1} (j-k-1) \varepsilon _j.$$
Since $2(k - i)+1 = l-2i+2$ and $2(j-k-1) +1 = 2j-l-2$, one obtains
$$2 \sum _{k<j-i} (\varepsilon _i - \varepsilon _j) 
+ \sum _{i = 1}^k (\varepsilon _i - \varepsilon _{i+k})
= 2 \rho_{{\rm A}_{\rm I}}.$$
Thus $\Psi$ satisfies (i).
The roots contained in $\Psi$ are those of the form $\varepsilon _i - \varepsilon _j$ with $j-i \geqslant k$.
No root of the system $\mathrm A_l$ is the sum of two of them, thus (iii) holds. 
Moreover the above roots generate $V$. Indeed:
\begin{itemize}
\item for $i < k$, one has $\varepsilon _i - \varepsilon _{i+1} = (\varepsilon _i - \varepsilon _{l+1}) - (\varepsilon _{i+1} - \varepsilon _{l+1})$, with $l+1 -i \geqslant l+1 - (i+1) \geqslant k$, 
\item for $i > k$, one has 
$\varepsilon _i - \varepsilon _{i+1} = (\varepsilon _1 - \varepsilon _{i+1}) - (\varepsilon _1 - \varepsilon _i)$,
with $i+1 - 1 \geqslant i- 1 \geqslant k$, 
\item and for $i=k$: $\varepsilon _k - \varepsilon _{k+1} = 
(\varepsilon _1 - \varepsilon _{k+1}) + (\varepsilon _k - \varepsilon _{l+1}) - (\varepsilon _1 - \varepsilon _{l+1})$.
\end{itemize}
Therefore (i)-(iii) are satisfied.

(2) Suppose $l = 2k$ with $k \geqslant 1$. Then:
$$\sum _{k<j-i} \varepsilon _i - \varepsilon _j = 
\sum _{i =1}^{k} (k+1-i)\varepsilon _i
- \sum _{j = k+2}^{l+1} (j-k-1) \varepsilon _j.$$
Since $2(k +1 - i)= l-2i+2$ and $2(j-k-1)= 2j-l-2$, one obtains that 
$2 \sum _{k<j-i} (\varepsilon _i - \varepsilon _j) = 2\rho_{{\rm A}_{\rm I}}$.
Therefore 
$$2\sum _{k<j-i<l} (\varepsilon _i - \varepsilon _j ) 
+ 2(\varepsilon _1 - \varepsilon _{k+1}) + 
2(\varepsilon _{k+1} - \varepsilon _{l+1}) = 2\rho_{{\rm A}_{\rm I}}.$$
Thus $\Psi$ satisfies (i).
The roots in $\Psi$ are those of the form $\varepsilon _i - \varepsilon _j$, with $k< j-i <l$, and the roots
$\varepsilon _1 - \varepsilon _{k+1}$, $\varepsilon _{k+1} - \varepsilon _{l+1}$. Apart $\varepsilon _1 - \varepsilon _{l+1}$
which is the sum of the last ones, no root of the system $\mathrm A_l$ is the sum of two of them. Thus (iii)$'$ holds.
In addition the above roots generate $V$. Indeed:
\begin{itemize}
\item $\varepsilon _1 - \varepsilon _{l+1} = (\varepsilon _1 - \varepsilon _{k+1}) + (\varepsilon _{k+1} - \varepsilon _{l+1})$,
\item for $i < k$, one has $\varepsilon _i - \varepsilon _{i+1} = (\varepsilon _i - \varepsilon _{l+1}) - 
(\varepsilon _{i+1} - \varepsilon _{l+1})$, with $l+1 -i \geqslant l+1 - (i+1) > k$, 
\item for $i > k+1$, one has 
$\varepsilon _i - \varepsilon _{i+1} = (\varepsilon _1 - \varepsilon _{i+1}) - (\varepsilon _1 - \varepsilon _i)$,
with $i+1 - 1 \geqslant i- 1 > k$, 
\item for $i =k$: $\varepsilon _k - \varepsilon _{k+1} = (\varepsilon _k - \varepsilon _{l+1}) - 
(\varepsilon _{k+1} - \varepsilon _{l+1})$,
\item for $i = k+1$: $\varepsilon _{k+1} - \varepsilon _{k+2} = (\varepsilon _1 - \varepsilon _{k+2}) - 
(\varepsilon _1 - \varepsilon _{k+1})$.
\end{itemize}
This concludes the proof. 
\end{proof}

\subsection{Proof of non-vanishing for semisimple groups}
We can finally put things together in order to provide a proof for Theorem \ref{introduction-theorem4}, hence a proof for Theorem \ref{introduction-theorem1} of the Introduction, in view of the reduction contained in \ref{ss - ss and QI}. 

\begin{proof}[Proof of Theorem \ref{introduction-theorem4}]
By Lemma \ref{lemma - ss OK}, it suffices to exhibit a suitable subset $\Psi$ of positive roots for any semisimple group.
This can be done separately for each connected component of the Dynkin diagram, which amounts to dealing with simple real Lie groups. 
The absolutely simple cases were treated by a case-by-case analysis in Subsections \ref{ss - linear algebra} to \ref{ss - type A}. 
The remaining cases correspond to the simple non absolutely simple groups, \emph{i.e.} simple complex Lie groups seen as real groups. 
In the latter cases, the root multiplicities are all equal to 2 since the groups are split over ${\bf C}$ and the root groups are all isomorphic to the real Lie group ${\bf C}$. 
Therefore the function $\rho_G$ in this case is twice the corresponding function for the split groups over ${\bf R}$ with the same root system, showing that the same subset $\Psi$ can be chosen, up to multiplying the coefficients $m_\beta$ by $2$ for each $\beta \in \Psi$. 
\end{proof}

\section {Cohomologies of semi-direct products}
\label{s - psd} 

This section relates the de Rham $L^p$-cohomology with the group $L^p$-cohomology. Our goal
is to tranfer to the setting of de Rham $L^p$-cohomology, a result issued from \cite{BR} about the group $L^p$-cohomology 
of semi-direct products, see Corollary \ref{coho-corollary}.
It leads to the key relation (\ref{introduction-eqn}) of the introduction.

The section is also an opportunity to advertise several incarnations of $L^p$-cohomology, 
and to present their properties in a synthetic way. It collects results issued from \cite{Pa95, SaSc, BR}
(see also \cite{Elek} for related results in the discrete group case).

\subsection{Asymptotic and group $L^p$-cohomologies}\label{coho-group}
~~~ The asymptotic $L^p$-cohomology of a metric space has been defined by Pansu in \cite{Pa95}.
Let $(X, d)$ be a metric space equipped with a Borel measure $\mu$. Suppose it satisfies the following ``bounded geometry''
condition. There exist non-decreasing functions $v, V: (0, +\infty) \to (0, +\infty)$, such for every ball 
$B(x, R) \subset X$ one has 
$$v(R) \leqslant \mu \bigl(B(x,R)\bigr) \leqslant V(R).$$ 
For $R >0$ and $k\in {\bf N}$, let
$$\Delta _R^{(k)} = \{(x_0,\dots , x_k) \in X^{k+1} ~\vert~ d(x_i, x_j) \leqslant R \,\, \hbox{\rm for all $i,j$}\}.$$
Let $AS^{p,k}(X)$ be the space of (classes of) measurable functions $f: X^{k+1} \to {\bf R}$ such that for every
$R >0$ one has 
$$N_R (f)^p:= \int _{\Delta _R^{(k)}} \bigl\vert f (x_0, \dots ,x_k) \bigr\vert^p d\mu(x_0) \dots d\mu (x_k) < +\infty.$$
We equip $AS^{p,k}(X)$ with the topology induced by the set of the semi-norms $N_R$ ($R>0$).

\begin{definition}\label{semi-direct-definition1} The \emph{asymptotic $L^p$-cohomology} of $X$ is the cohomology of the complex 
$AS^{p,0}(X) \stackrel{\delta _0}{\to} AS^{p,1}(X) \stackrel{\delta _1}{\to} AS^{p,2}(X) \stackrel{\delta _2}{\to} \dots$, 
where the $\delta _k$'s are defined by:
\begin{equation}\label{coho-differential} 
(\delta _k f)(x_0, \dots , x_{k+1}) = \sum _{i =0}^{k+1} (-1)^{i} f(x_0 , \dots , \hat{x_i}, \dots , x_{k+1}).
\end{equation}
The \emph{reduced} asymptotic $L^p$-cohomology is defined similarly.
They are denoted by $L^p {\rm H}_{\mathrm{AS}}^* (X)$
and $L^p \overline{{\rm H}_{\mathrm{AS}}^*} (X)$ respectively.
\end{definition}

\begin{theorem}
\label{coho-theorem1}
\cite[Section 2]{Pa95} 
Let $X$ and $Y$ be metric spaces.
Assume that each of them admits a Borel measure with respect to which it is of bounded geometry (as defined above). 
Let $F: X \to Y$ be a quasi-isometry. 
Then $F$ induces a homotopy equivalence\footnote{All the maps occuring in homotopies are supposed to be continuous.} 
between the complexes $AS^{p,*}(Y)$ and $AS^{p,*}(X)$, 
and a canonical isomorphism of graded topological vector spaces $F^*: L^p {\rm H}_{\mathrm{AS}}^* (Y) \to L^p {\rm H}_{\mathrm{AS}}^* (X)$.
In particular $F^*$ depends only on the bounded perturbation class of $F$. The same holds in reduced cohomology.
\end{theorem}
See also \cite{Genton, SaSc} for a more detailed proof.

We now turn our attention to the continuous group cohomology, 
see \cite[Chap.\! IX]{BW} or \cite{Gui} for more details. 
 
Let $G$ be a locally compact second countable group. It admits a left-invariant proper metric defining its topology, see \emph{e.g.}\!
\cite[Struble Th.\! 2.B.4]{CdH}. Let $(\pi, V)$ be a \emph{topological $G$-module} \emph{i.e.}~ a Hausdorff locally convex 
vector space over ${\bf R}$ on which $G$ acts via a continuous representation $\pi$. 
We denote by $V^G \subset V$ the subspace of $\pi(G)$-invariant vectors.
For $k \in {\bf N}$, let $C^k(G, V)$ be the space of continuous maps from $G^{k+1}$ to $V$ equipped with the compact open topology. 
Then $C^k(G,V)$ is a topological $G$-module by means of the following action:
for $g, x_0, \dots , x_k \in G$,
$$(g \cdot f) (x_0, \dots , x_k) = \pi(g) \bigl(f(g^{-1} x_0, \dots , g^{-1} x_k)\bigr).$$
Consider the following complex of invariants:
$$C^0(G,V)^G \stackrel{\delta _0}{\to} C^1(G,V)^G \stackrel{\delta _1}{\to} C^2(G,V)^G \stackrel{\delta _2}{\to} \dots $$
where the $\delta _k$'s are defined as in (\ref{coho-differential}).
The \emph{continuous cohomology of $G$ with coefficients in $(\pi ,V)$} is the cohomology of this complex, it will be denoted
by ${\rm H}_{\mathrm{ct}}^*(G, V)$. Similarly is defined the reduced cohomology $\overline{{\rm H}_{\mathrm{ct}}^*}(G, V)$.

\begin{definition}
\label{semi-direct-definition2} 
Let $\mathcal H$ be a left-invariant Haar measure on $G$. The \emph{group $L^p$-cohomology} of $G$
is the continuous cohomology of $G$, with coefficients in
the right-regular representation of $G$ on $L^p(G, \mathcal H)$, \emph{i.e.}~ the representation defined by
$$\bigl(\pi(g) u\bigr)(x) = u(xg) ~~\mathrm{~~for~~}~~ u \in L^p(G, \mathcal H) ~~\mathrm{~~and~~}~~ g, x \in G.$$
It will be denoted by ${\rm H}_{\mathrm{ct}}^*(G, L^p(G))$.
The \emph{reduced} group $L^p$-cohomology of $G$ is defined similarly and is denoted by 
$\overline{{\rm H}_{\mathrm{ct}}^*}(G, L^p(G))$. 
\end{definition}
Observe that the right-regular representation on $L^p(G, \mathcal H)$
is isometric if and only if $G$ is unimodular.
\begin{theorem}\label{coho-theorem2}\cite[Theorem 10] {SaSc}, 
\cite[Theorem 3.6]{BR}
Suppose $G$ is a locally compact second countable topological group equipped with a left-invariant proper metric.
Then the complex of invariants $C^*(G, L^p(G))^G$ is canonically homotopy equivalent to $AS^{p,*}(G)$.
In consequence, there exists a canonical isomorphism of gradued topological vector spaces 
${\rm H}_{\mathrm{ct}}^*(G, L^p(G)) \simeq L^p {\rm H}_{\mathrm{AS}}^*(G)$.
The same holds for the reduced cohomology. 
\end{theorem}
We notice that the above isomorphism admits the following property. Suppose we are given an isomorphism $\varphi: G_1 \to G_2$ of
topological groups as above. It induces isomorphisms of complexes:
$$\varphi _{\mathrm{ct}}^*: C^k\bigl(G_2, L^p(G_2)\bigr)^{G_2} \to C^k\bigl(G_1, L^p(G_1)\bigr)^{G_1}$$
$$\mathrm{and}~~~~\varphi _{\mathrm{AS}}^*: AS^{p,k}(G_2) \to AS^{p,k}(G_1),$$
$$\mathrm{defined~by}~~~~\varphi _{\mathrm{ct}}^*(f)(x_0, \dots, x_k; x) = f\bigl(\varphi(x_0), \dots, \varphi(x_k); \varphi (x)\bigr),$$ 
$$\mathrm{and}~~~~\varphi _{\mathrm{AS}}^*(f)(x_0, \dots, x_k) = f\bigl(\varphi(x_0), \dots, \varphi(x_k)\bigr).$$
Then the isomorphisms ${\rm H}_{\mathrm{ct}}^*(G_i, L^p(G_i)) \simeq L^p {\rm H}_{\mathrm{AS}}^*(G_i)$, together with the induced isomorphisms
$\varphi _{\mathrm{ct}}^*: {\rm H}_{\mathrm{ct}}^*(G_2, L^p(G_2)) \to {\rm H}_{\mathrm{ct}}^*(G_1, L^p(G_1))$,
$\varphi _{\mathrm{AS}}^*: L^p {\rm H}_{\mathrm{AS}}^*(G_2) \to L^p {\rm H}_{\mathrm{AS}}^*(G_1)$, form a commutative diagram.

\subsection{Asymptotic and de Rham $L^p$-cohomologies}\label{coho-asymp}
Let $M$ be a $C^\infty$ complete Riemannian manifold. Unlike its asymptotic $L^p$-cohomology,
its simplicial $L^p$-cohomology is not invariant by quasi-isometry (e.g.\! for compact manifolds it is isomorphic to the standard de Rham 
cohomology). 
For that reason we restrict ourself to manifolds which are ``uniformly diffeomorphic to ${\bf R}^D$''. 

\begin{definition}
\label{semi-direct-definition} 
Let $B = B(0,1)$ be the unit open ball in ${\bf R}^D$.
 A manifold $M$ is said \emph{uniformly diffeomorphic to ${\bf R}^D$},
if there exist functions $\rho, \lambda: [0, +\infty) \to [1, +\infty)$, and for every $m \in M$ a $C^\infty$-diffeomorphism
$\varphi _m: B \to M$, such that
\begin{itemize}
\item for every $m \in M$ and $R >0$, there exists $r \in (\frac{1}{2}, 1)$ with 
$$B(m,R) \subset \varphi _m \bigl(B(0,r)\bigr) \subset B\bigl(m, \rho(R)\bigr),$$
\item and $\varphi _m \restr _{B(0,r)}$ is $\lambda(R)$-bi-Lipschitz.
\end{itemize}
\end{definition}

For example a Riemannian manifold which is diffeomorphic to ${\bf R}^D$, and which admits a cocompact group of isometries, is uniformly diffeomorphic to ${\bf R}^D$.

\begin{theorem}
\label{coho-theorem3}\cite{Pa95} 
Let $M$ be a $C^\infty$ Riemannian manifold which is uniformly diffeomorphic to ${\bf R}^D$.
There exists a homotopy equivalence between the 
complexes $AS^{p, *}(M)$ and $\Omega^{p,*}(M)$. It induces a canonical isomorphism of gradued topological vector spaces $L^p {\rm H}_{\mathrm{AS}}^*(M)
\simeq L^p {\rm H}_{\mathrm{dR}}^*(M)$. The same holds for reduced cohomology.
\end{theorem}

The proof that Pansu gives in \cite{Pa95} goes through the simplicial $\ell^p$-cohomology. 
A more direct proof is presented in Appendix \ref{app - asymptotic and de Rham} below.

We notice that the isomorphism in Theorem \ref{coho-theorem3} above admits the following property.
Suppose we are given a bi-Lipschitz $C^\infty$ diffeomorphism $\varphi: M_1 \to M_2$ between Riemannian manifolds as above.
Then the isomorphisms 
$L^p {\rm H}_{\mathrm{dR}}^*(M_i) \simeq L^p {\rm H}_{\mathrm{AR}}^*(M_i)$, together with the induced isomorphisms 
$\varphi^*_{\mathrm{dR}}: L^p {\rm H}_{\mathrm{dR}}^*(M_2) \to L^p {\rm H}_{\mathrm{dR}}^*(M_1)$,
$\varphi _{\mathrm{AS}}^*: L^p {\rm H}_{\mathrm{AS}}^*(M_2) \to L^p {\rm H}_{\mathrm{AS}}^*(M_1)$,
form a commutative diagram.

\subsection{Semi-direct products}
We transfer to de Rham $L^p$-cohomology a result about group $L^p$-cohomology of semi-direct products,
see Corollary \ref{coho-corollary}. 
This leads to relation (\ref{introduction-eqn}) of the introduction.

Let $(V, \Vert \cdot \Vert _V)$ be a separable normed space, let $X$ be a locally compact second countable topological space
endowed with a Radon measure $\mu$, 
and let $p \in (1, +\infty)$. 
We denote by $L^p(X,V)$ the normed space consisting of the (classes of) mesurable maps 
$f: X \to V$ such that
$$\Vert f \Vert _{L^p(X,V)}^p := \int _X \bigl\Vert f(x) \bigr\Vert _V^p ~d\mu (x)< +\infty.$$
Let $G$ be a locally compact second countable group. Suppose it decomposes as a semi-direct product 
$G = Q \ltimes H$, with $Q$, $H$ closed subgroups and the standard multiplicative law 
$$(q_1, h_1)\cdot(q_2,h_2) = (q_1q_2, q_2^{-1}h_1q_2h_2).$$ 
Let $\mathcal H_Q$, $\mathcal H_H$ be left-invariant Haar measures on $Q$, $H$
respectively. Then $\mathcal H_G:= \mathcal H_Q \times \mathcal H_H$ is a left-invariant Haar measure on $G$. 
One has:

\begin{theorem}
\label{coho-theorem4}
\cite[Corollary 5.5]{BR} 
Let $G = Q \ltimes H$ as above.
Assume that the complex of invariants $C^*(H, L^p(H))^H$ is 
homotopically equivalent to a complex of Banach spaces. Suppose also
that there exists $n \in {\bf N}$, such that ${\rm H}_{\mathrm{ct}}^k(H, L^p(H)) = 0$ for $0 \leqslant k < n$ and such that 
${\rm H}_{\mathrm{ct}}^n(H, L^p(H))$ is Hausdorff.
Then ${\rm H}_{\mathrm{ct}}^k(G, L^p(G)) = 0$ for $0 \leqslant k < n$ and there is a linear isomorphism
\begin{equation}\label{coho-semidirect} 
{\rm H}_{\mathrm{ct}}^n \bigl(G, L^p(G)\bigr) \simeq L^p\Bigl(Q, {\rm H}_{\mathrm{ct}}^n \bigl(H, L^p(H)\bigr)\Bigr)^Q,
\end{equation}
where the $Q$-action on $L^p(Q, {\rm H}_{\mathrm{ct}}^n(H, L^p(H)))$ is by right multiplication on itself 
and by conjugacy on ${\rm H}_{\mathrm{ct}}^n(H, L^p(H))$;
in other words it is the action induced by
$$(q \cdot f)(y)( x_0,\dots , x_n; x) = f(yq)(q^{-1} x_0 q, \dots , q^{-1} x_n q; q^{-1} x q),$$
for every $q, y \in Q$, $f: Q \to C^n(H, L^p(H))$ and $x_0,\dots ,x_n, x \in H$.
\end{theorem}

For Lie groups diffeomorphic to ${\bf R}^D$, one gets the following result which implies the relation 
(\ref{introduction-eqn}) in the introduction.

\begin{corollary}
\label{coho-corollary} 
Let $G$ be a Lie group diffeomorphic to ${\bf R}^D$. 
Suppose it decomposes as $G = Q \ltimes H$ with $Q$, $H$ closed subgroups and $H$ diffeomorphic to ${\bf R}^d$ for some $d$. 
Suppose that there exists $n \in {\bf N}$, such that $L^p{\rm H}_{\mathrm{dR}}^k(H) = 0$ for $0 \leqslant k < n$ and such that 
$L^p {\rm H}_{\mathrm{dR}}^n(H)$ is Hausdorff.
Then $L^p{\rm H}_{\mathrm{dR}}^k(G) = 0$ for $0 \leqslant k < n$ and there is a linear isomorphism
\begin{equation}
\label{coho-semidirect2} 
L^p{\rm H}_{\mathrm{dR}}^n(G) \simeq L^p\bigl(Q, L^p{\rm H}_{\mathrm{dR}}^n(H)\bigr)^Q,
\end{equation}
where the $Q$-action on $L^p(Q, L^p{\rm H}_{\mathrm{dR}}^n(H))$ is by right multiplication on itself and by conjugacy on $L^p{\rm H}_{\mathrm{dR}}^n(H)$;
in other words, it is induced by the action 
$$(q \cdot f)(y) = C^*_{q^{-1}} \bigl(f(yq)\bigr),$$
for every $q, y \in Q$, $f: Q \to \Omega^{p,n}(H)$.
\end{corollary}

\begin{proof}
Since $G$ and $H$ are homogeneous and respectively diffeomorphic to ${\bf R}^D$ and ${\bf R}^d$, they are uniformly diffeomorphic to ${\bf R}^D$ and ${\bf R}^d$.
By applying successively Theorems \ref{coho-theorem2} and \ref{coho-theorem3}, one sees that the complexes $C^*(G, L^p(G))^G$ and $\Omega^{p,*}(G)$ are homotopy equivalent. 
The same holds for $C^*(H, L^p(H))^H$ and $\Omega^{p,*}(H)$. 
In particular $C^*(H, L^p(H))^H$ is homotopy equivalent to a complex of Banach spaces, and thus Theorem \ref{coho-theorem4} applies. 

It remains to relate the expressions of the $Q$-action on ${\rm H}_{\mathrm{ct}}^n(H, L^p(H))$ and on $L^p{\rm H}_{\mathrm{dR}}^n(H)$. 
For that we need to find the expressions the conjugacy $C_q^*: L^p{\rm H}_{\mathrm{dR}}^n(H) \to L^p{\rm H}_{\mathrm{dR}}^n(H)$ by $q \in Q$, when transformed by the successive isomorphisms
$$ L^p{\rm H}_{\mathrm{dR}}^n(H) \simeq L^p{\rm H}_{\mathrm{AS}}^n(H) \simeq {\rm H}_{\mathrm{ct}}^n \bigl(R, L^p(H)\bigr).$$
It is done by using their functional properties described right after Theorems \ref{coho-theorem3} and 
\ref{coho-theorem2}.
\end{proof}

\begin{remark} 
The linear isomorphism (\ref{coho-semidirect}) in Theorem \ref{coho-theorem4} is the composition of the following two isomorphisms
$${\rm H}_{\mathrm{ct}}^n \bigl(G, L^p(G)\bigr) 
\stackrel{\varphi_1}{\simeq} 
{\rm H}_{\mathrm{ct}}^n \bigl(H, L^p(G)\bigr)^Q 
\stackrel{\varphi_2}{\simeq}
L^p\Bigl(Q, {\rm H}_{\mathrm{ct}}^n \bigl(H, L^p(H)\bigr)\Bigr)^Q.$$
The second one is a topological isomorphism \cite[Proposition 5.2]{BR}; it comes from a Fubini type argument. 
The first one comes from the Hochschild-Serre spectral sequence \cite[Theorem IX.4.3]{BW}.
When $G$ is countable, the continuous cohomology ${\rm H}_{\mathrm{ct}}^n(G, L^p(G))$ coincides with the standard
one ${\rm H}^n(G, \ell^p(G))$. 
In this case it is known that $\varphi _1$ equals the restriction map \cite[Theorem III.2]{HS}. 
We suspect that the equality between $\varphi _1$ and the restriction map holds in general under the assumptions of Theorem \ref{coho-theorem4}. 
Since the restriction map is continuous, this would imply that ${\rm H}_{\mathrm{ct}}^n(G, L^p(G))$ is Hausdorff and that the isomorphisms (\ref{coho-semidirect}) and (\ref{coho-semidirect2}) are canonical and Banach.
\end{remark}

%

\appendix

\section{On asymptotic and de Rham $L^p$-cohomology}
\label{app - asymptotic and de Rham} 
The goal of the appendix is to give a direct proof of Pansu's Theorem \ref{coho-theorem3}, that we restate below for commodity (see Definition \ref{semi-direct-definition} for the notion of \emph{uniformly diffeomorphic to ${\bf R}^D$}):

\begin{theorem}
\label{appendix-theorem} 
Let $M$ be a $C^\infty$ Riemannian manifold which is uniformly diffeomorphic to ${\bf R}^D$.
There exists a homotopy equivalence\footnote{Recall that all the maps occuring in homotopies are supposed to be continuous.} between the 
complexes $AS^{p, *}(M)$ and $\Omega^{p,*}(M)$. 
It induces a canonical isomorphism of gradued topological vector spaces $L^p {\rm H}_{\mathrm{AS}}^*(M)
\simeq L^p {\rm H}_{\mathrm{dR}}^*(M)$. The same holds for reduced cohomology.
\end{theorem}

The following standard notion will serve repeatedly in the sequel.

\begin{definition} Let $(A^*, d)$ be a complex of topological vector spaces, let $(B^*, d) \subset (A^*, d)$ be a subcomplex, and let denote
by $i: B^* \to A^*$ the inclusion map. One says that $A^*$ \textit{retracts by deformation onto} $B^*$, if there exists a continuous linear
map $r: A^* \to B^*$ such that $r \circ d = d \circ r$, $r \circ i = \mathrm{id}$, and $i \circ r$ is homotopic to $\mathrm{id}$.
\end{definition} 

When $A^*$ retracts by deformation onto $B^*$, the inclusion map $i: B^* \to A^*$ induces canonically isomorphisms 
of gradued topological vector spaces
$\mathrm{H}^*(B^*) \simeq \mathrm{H}^*(A^*)$ and $\overline{\mathrm{H}^*}(B^*) \simeq \overline{\mathrm{H}^*}(A^*)$.

The method for proving Theorem \ref{appendix-theorem} is a variant of the double complex proof 
of the isomorphism between de Rham and Cech cohomologies
(see \emph{e.g.}\! \cite[Theorem 8.9 and Proposition 9.5]{BT82}). This method is used in \cite{Pa95} to show the equivalence between the
de Rham and the simplicial
$L^p$-cohomologies. It is based on the following general lemma:

\begin{lemma}\label{appendix-lemma0}
Let $(C^{k, \ell}, d', d'')_{(k,\ell) \in \mathbf N^2}$ be a double complex of topological vector spaces, with 
$d'\circ d'' +d''\circ d' =0$. Suppose that for every $\ell \in \mathbf N$,
the complex $(C^{*, \ell}, d')$ retracts by deformation onto the subcomplex that is null in degrees at least $1$ and equal
to $E^\ell:= \Ker d'\vert _{C^{0,\ell}}$ in degree $0$. Then the complex $(D^*, d_D)$, defined
by $D^m = \bigoplus _{k + \ell = m} C^{k, \ell}$ and $d_D = d' + d''$, retracts by deformation onto its subcomplex $(E^*, d'')$.
\end{lemma}

This kind of result is well-known from the specialists (see \emph{e.g.}\! \cite[Proposition 9.5 and Remark p.\! 104]{BT82}). 
It is stated in this form in \cite[Lemme 5]{Pa95}, with a sketch of proof.
A detailed proof appears in \cite[Lemma 2.2.1]{Sequeira}.

\subsection{Definition of the double complex $C^{*,*}$}

Let $M$ be a $C^\infty$ complete Riemannian manifold. We denote by $\Omega^{p,k} _{\mathrm{loc}}(M)$ the space of (measurable) $k$-differential forms on $M$
that belong to $\Omega^{p,k} (U)$ for every relatively compact open subset $U \subset M$. 
The $\Omega^{p,k}$-norm of the restriction of $\omega$ to $U$ is denoted by $\Vert \omega \Vert _{U}$ for simplicity.
Equipped with the set of semi-norms $\Vert \cdot \Vert _{U}$, where $U \subset M$ is open relatively compact, 
the space $\Omega^{p,k} _{\mathrm{loc}}(M)$
is a separable Fr\'echet space. 

For every couple $(k, \ell) \in \mathbf N^2$, every $R>0$ and every measurable map $f: M^{\ell +1} \to \Omega^{p,k} _{\mathrm{loc}}(M)$,
we define the semi-norm $N_R (f)$ by
$$N_R (f)^p = \int _{\Delta _R^{(\ell)}} \bigl\Vert f(m_0,\dots , m_{\ell}) \bigr\Vert _{B(m_0,R)}^p 
d\mathrm{vol}^{\ell +1}(m_0,\dots , m_{\ell}),$$
where $\mathrm{vol}^{\ell +1}$ denotes the product measure on $M^{\ell +1}$. We remark that changing $B(m_0, R)$ by $B(m_i,R)$
in the definition of $N_R(f)$, leads to an equivalent family of semi-norms; indeed one has $B(m_i,R) \subset B(m_j, 2R)$
when $(m_0, \dots, m_{\ell}) \in \Delta _R^{(\ell)}$.

Let $C^{k, \ell}$ be the topological vector spaces of (the classes of) the measurable maps 
$f: M^{\ell +1} \to \Omega^{p,k} _{\mathrm{loc}}(M)$
such that $N_R (f) < +\infty$ for every $R>0$.
We define $d': C^{k, \ell} \to C^{k+1, \ell}$ and $d'': C^{k, \ell} \to C^{k, \ell +1}$ by
$$d'f = (-1)^\ell d \circ f ~~~\mathrm{~~and~~}~~~ d''f = \delta f,$$
where $d$ is the de Rham differential operator and $\delta$ is the discrete operator defined in (\ref{coho-differential}).
Then $C^{*,*} = (C^{k, \ell}, d', d'')_{(k, \ell) \in \mathbf N^2}$ is a double complex of topological vector spaces. 
It satisfies $d' \circ d'' + d'' \circ d' =0$.

The complex $C^{*,*}$ interpolates between the complexes $AS^{p,*}(M)$ and $\Omega^{p, *}$. Indeed:
\begin{proposition}\label{appendix-proposition1}
There are canonical isomorphisms of topological complexes
$\Ker d' \restr _{C^{0, *}} \simeq AS^{p, *}(M)$ and $\Ker d'' \restr _{C^{*, 0}} \simeq \Omega^{p, *}(M)$.
\end{proposition}

\begin{proof} 
(1) One has $f \in \Ker d' \restr _{C^{0, \ell}}$ if and only if $f: M^{\ell +1} \to \Omega _{\mathrm{loc}}^{p, 0}(M)$ 
satisfies $d \circ f =0$, 
\emph{i.e.}~ $f(m_0, \dots, m_\ell)$ is a constant function for a.a.\! $(m_0, \dots, m_\ell) \in M^{\ell +1}$. 
Moreover when $f(m_0, \dots, m_\ell)$ is a constant function, one has 
$$\bigl\Vert f(m_0, \dots, m_\ell) \bigr\Vert _{B(m_0,R)}^p = 
\bigl\vert f(m_0, \dots, m_\ell) \bigr\vert^p  \mathrm{vol}\bigl(B(m_0, R)\bigr).$$
Therefore $\Ker d' \restr _{C^{0, *}} \simeq AS^{p, *}(M)$.

(2) One has $f \in \Ker d'' \restr _{C^{k, 0}}$ if and only if $f: M \to \Omega _{\mathrm{loc}}^{p, k}(M)$ 
satisfies $\delta f =0$ \emph{i.e.}~
$f$ is a constant map. Let $\omega$ be its constant value.
Since $M$ is uniformly diffeomorphic to ${\bf R}^D$, it admits bounded geometry
in the sense of Subsection \ref{coho-group}; there exist functions $v, V: (0, +\infty) \to (0, +\infty)$
such that such for every ball 
$B(m, R) \subset M$ one has 
$$v(R) \leqslant \mathrm{vol}\bigl(B(m,R)\bigr) \leqslant V(R).$$ 
By Fubini one has for every $R>0$ 
$$v(R) \Vert \omega \Vert^p _{\Omega^{p,k}} \leqslant \int_M \Vert \omega \Vert^p _{B(m, R)} d\mathrm{vol}(m)
\leqslant V(R) \Vert \omega \Vert^p _{\Omega^{p,k}}.$$
Therefore $\Ker d'' \restr _{C^{*, 0}} \simeq \Omega^{p, *}(M)$.
\end{proof}

\subsection{Homotopy type of the columns}

\begin{proposition} \label{appendix-proposition2} For every $\ell \in \mathbf N$, the complex $(C^{*, \ell}, d')$ 
retracts by deformation onto the subcomplex $(\Ker d' \restr _{C^{0, \ell}} \to 0 \to 0 \to \dots)$.
\end{proposition}
Its proof is postened at the end of the subsection. A crucial ingredient is the following lemma
issued from \cite[Section 4]{IL93}.

\begin{lemma}\label{appendix-lemma1}
Let $B = B(0,1)$ be the unit Euclidean open ball in ${\bf R}^D$. Let $h \in C^\infty (B)$ be non-negative,
supported on $B(0, \frac{1}{2})$, and normalized so that $\int _B h(x)dx =1$. There exists a continuous
operator $T: \Omega _{\mathrm{loc}}^{p,k}(B) \to \Omega _{\mathrm{loc}}^{p,k-1}(B)$ with the following homotopy
properties
\begin{enumerate}
\item $d \circ T + T \circ d = \mathrm{id}$ when $k \geqslant 1$,
\item $(T \circ d)(f) = f -\int _B f(x)h(x) dx$, for $f \in \Omega _{\mathrm{loc}}^{p,0}(B)$.
\item For every $r \in (\frac{1}{2}, 1)$, $k \geqslant 1$ and $\omega \in \Omega _{\mathrm{loc}}^{p, k}(B)$, one has
$$\Vert T \omega \Vert _{B(0,r)} \leqslant C  \Vert \omega \Vert _{B(0,r)},$$ where $C$ is a constant 
which depends only on the dimension $D$.
\end{enumerate}
\end{lemma}
\begin{proof} This is precisely done in \cite[Section 4]{IL93}. We recall the construction for convenience. 
First, for every $y \in B$, one defines an operator $K_y: \Omega^k(B) \to \Omega^{k-1}(B)$ by H.\! Cartan's
formula 
$$(K_y \omega)(x; v_1, \dots, v_{k-1}) = \int _0^1 t^{k-1} \omega\bigl(y+t(x-y); x-y, v_1, \dots, v_{k-1}\bigr)~dt.$$
It satisfies $d \circ K_y + K_y \circ d = \mathrm{id}$ when $k\geqslant 1$, and $(K_y \circ d)(f) = f(y)$ when $k=0$.
Then one averages $K_y$ over all $y \in B$, to define the operator $T: \Omega^k(B) \to \Omega^{k-1}(B)$:
$$T\omega = \int _B (K_y\omega)h(y)dy.$$
Clearly it satisfies the homotopy relations (1) and (2) of the Lemma. A bit of analysis is required
to see that it extends to an operator from $\Omega _{\mathrm{loc}}^{p,k}(B)$ to $\Omega _{\mathrm{loc}}^{p,k-1}(B)$,
and to show that it satisfies the property (3) of Lemma \ref{appendix-lemma1} -- 
see Inequality (4.15) and Lemma 4.2 in \cite{IL93}.
\end{proof}

By assumption $M$ is uniformly diffeomorphic to ${\bf R}^D$. Thus, for every $m \in M$, there is a diffeomorphism 
$\varphi_m: B \to M$ with controlled geometry, see Definition \ref{semi-direct-definition}.

We remark that we can (and will) assume that the map 
$$ M \times B \to M, ~~(m, x) \mapsto \varphi _m(x)$$ 
is measurable.
Indeed one can always transform the family $\{\varphi _m\}_{m \in M}$ in such a way that the resulting map 
$M \to C^\infty (B, M), ~~ m \mapsto \varphi _m$
is piecewise constant.
\begin{lemma}\label{appendix-lemma1bis}
For $f: M^{\ell +1} \to \Omega^{p,k} _{\mathrm{loc}}(M)$ and $(m_0, \dots, m_\ell) \in M^{\ell +1}$, 
let 
$$(Hf)(m_0, \dots, m_\ell) = 
\bigl( (\varphi _{m_0}^{-1})^* \circ T \circ \varphi _{m_0}^* \bigr)\bigl(f(m_0, \dots, m_\ell)\bigr),$$
where $T$ is the homotopy operator in Lemma \ref{appendix-lemma1}.
This defines a continuous operator $H: C^{k,\ell} \to C^{k-1,\ell}$, which satisfies the following homotopy
relations 
\begin{itemize} 
\item $d' \circ H + H \circ d' = \mathrm{id}$ for $k \geqslant 1$,
\item $H \circ d' = \mathrm{id} - \psi$ for $k=0$, where 
$$(\psi f)(m_0, \dots, m_\ell) = \int _B f\bigl(m_0, \dots, m_\ell; \varphi _{m_0}(x)\bigr)h(x)dx.$$
\end{itemize}
\end{lemma}
\begin{proof}
(1) First, since $(m, x) \mapsto \varphi _m(x)$ is measurable on $M \times B$, the map $Hf$ is measurable on $M^{\ell +1}$.

(2) For $\omega \in \Omega^{p,k} _{\mathrm{loc}}(M)$, $m \in M$ and $R>0$, with the notations of 
Definition \ref{semi-direct-definition} and Lemma \ref{appendix-lemma1}, one has:
\begin{align*}
&\bigl\Vert \bigl( (\varphi _{m}^{-1})^* \circ T \circ \varphi _{m}^* \bigr) (\omega)\bigr\Vert _{B(m,R)} 
= \bigl\Vert (\varphi _{m}^{-1})^* \bigl( (T \circ \varphi _{m}^* )(\omega)\bigr) \bigr\Vert _{B(m,R)}\\ 
\leqslant~ & \bigl\Vert (\varphi _{m}^{-1})^* \bigl( (T \circ \varphi _{m}^* )(\omega)\bigr) \bigr\Vert _{\varphi _m (B(0,r))}
\leqslant \lambda(R)^{k+\frac{p}{D}} \bigl\Vert (T \circ \varphi _{m}^* )(\omega) \bigr\Vert _{B(0,r)} \\
\leqslant~ & C \lambda(R)^{k+\frac{p}{D}} \bigl\Vert \varphi _{m}^* (\omega) \bigr\Vert _{B(0,r)}
\leqslant C  \lambda(R)^{2k+1+\frac{2p}{D}} \Vert \omega \Vert _{\varphi _m (B(0,r))} \\
\leqslant~ & C  \lambda(R)^{2k+1+\frac{2p}{D}} \Vert \omega \Vert _{B(m, \rho(R))}.
\end{align*}
Therefore 
$N_R(Hf) \leqslant C  \lambda(R)^{2k+1+\frac{2p}{D}} N_{\rho(R)}(f)$, and thus $H$ maps $C^{k,\ell}$ to $C^{k-1,\ell}$ continuously.

(3) The homotopy relations follows easily from those in Lemma \ref{appendix-lemma1}.
\end{proof}

\begin{proof}[Proof of Proposition \ref{appendix-proposition2}]
We keep the notations of Lemma \ref{appendix-lemma1bis}. Recall that the subspace $\Ker d' \restr _{C^{0, \ell}}$ is described
in the proof on Proposition \ref{appendix-proposition1}. Define a retraction 
$r: C^{*,\ell} \to (\Ker d' \restr _{C^{0, \ell}} \to 0 \to 0 \to \dots)$ by letting $r = \psi$ on $C^{0, \ell}$, and $r = 0$ on $C^{k, \ell}$ 
when $k \geqslant 1$. Clearly it commutes with $d'$. Since $\int _B h(x)dx = 1$, one has $r \circ i = \mathrm{id}$. Finally $H$ is a
homotopy between $i \circ r$ and $\mathrm{id}$. 
\end{proof}


\subsection{Homotopy type of the rows}
\begin{proposition} \label{appendix-proposition3} For every $k \in \mathbf N$, the complex $(C^{k, *}, d'')$ 
retracts by deformation onto the subcomplex $(\Ker d'' \restr _{C^{k, 0}} \to 0 \to 0 \to \dots)$.
\end{proposition}
Again we start with two lemmata.
\begin{lemma}\label{lemma-chi} There exists, for every $m \in M$, a function $\chi _m: M \to [0, +\infty)$ which enjoys the following properties
:
\begin{enumerate}
\item its support is contained in $B(m,1)$,
\item the maps $(m, m') \mapsto \chi_m (m')$ and $(m, m') \mapsto (d\chi_m)(m')$ are $C^\infty$ 
and bounded on $M \times M$,
\item $\int _M \chi_m ~d\mathrm{vol}(m) =1$.
\end{enumerate}
\end{lemma}

\begin{proof} Consider a $C^\infty$ non-negative function $\Phi$ on $M \times M$, which is equal to $1$
on $\{d(m,m') \leqslant 1/2\}$, and $0$ on $\{d(m,m') \geqslant 1\}$. Let
$$\chi_m (m'):= \frac{\Phi (m, m')}{\int _M \Phi (m, m') ~d\mathrm{vol}(m)}.$$ 
It satisfies the excepted properties because $M$ is uniformly diffeomorphic to 
${\bf R}^D$.
\end{proof} 

\begin{lemma}\label{appendix-lemma2}
For $f: M^{\ell +1} \to \Omega _{\mathrm{loc}}^{p,k}(M)$ and $(m_0, \dots, m_{\ell -1}) \in M^{\ell}$, let
$$(Kf)(m_0, \dots, m_{\ell -1}) = \int _M \chi_m \cdot f(m, m_0, \dots, m_{\ell -1})~ d\mathrm{vol}(m).$$
This defines a continuous operator $K: C^{k, \ell} \to C^{k, \ell-1}$, which satisfies the following homotopy relations:
\begin{itemize}
\item $d'' \circ K + K \circ d'' = \mathrm{id}$ when $\ell \geqslant 1$,
\item $(K \circ d'')(f) = f - \int _M \chi _m \cdot f(m)~ d\mathrm{vol}(m)$ when $\ell =0$.
\end{itemize}
\end{lemma}

\begin{proof} We divide it into few steps.

(1) The map $K$ is a continuous linear from $C^{k, \ell}$ to $C^{k, \ell-1}$:

For $f \in C^{k, \ell}$, thanks to properties (1), (2) in Lemma \ref{lemma-chi}, one has:
\begin{align*}
&~\bigl\Vert (Kf)(m_0, \dots, m_{\ell -1}) \bigr\Vert _{B(m_0, R)} =\\
= &~\bigl\Vert \int _M \chi_m \cdot f(m, m_0, \dots, m_{\ell -1}) ~d\mathrm{vol}(m) 
\bigr\Vert _{B(m_0, R)} \\
\leqslant & 
\int _M \bigl\Vert \chi_m \cdot f(m, m_0, \dots, m_{\ell -1}) \bigr\Vert _{B(m_0, R)} ~d\mathrm{vol}(m)\\
\leqslant &~ C \cdot \int _{B(m_0, R+1)} 
\bigl\Vert f(m, m_0, \dots, m_{\ell -1}) \bigr\Vert _{B(m_0, R)} ~d\mathrm{vol}(m),
\end{align*}
where $C$ is a constant which depends only on the upper bounds in Lemma \ref{lemma-chi}(2). 
Since $M$ has bounded geometry the volume of the ball
$B(m_0, R+1)$ is bounded by above by a function of $R$ only. 
Moreover the relations $m \in B(m_0, R+1)$ and $(m_0, \dots, m_{\ell -1}) \in \Delta _R^{(\ell -1)}$ imply that 
$(m, m_0, \dots, m_{\ell -1}) \in \Delta _{2R +1}^{(\ell)}$ and that $B(m_0, R) \subset B(m, 2R+1)$. 
These properties, in combination with H\"older inequality and Fubini, yield the existence of
a function $\phi: (0, +\infty) \to (0, +\infty)$ such that 
$N_R(Kf) \leqslant \phi(R) \cdot N_{2R+1} (f)$.

\medskip

(2) For $m \in M$ and $f: M^{\ell +1} \to \Omega _{\mathrm{loc}}^{p,k}(M)$, put 
$$(k_m f)(m_0, \dots, m_{\ell -1}) = f(m, m_0, \dots, m_{\ell-1}),$$ so that
one can write $Kf = \int_M \chi_m \cdot (k_m f) ~d\mathrm{vol}(m)$.

It is an easy and standard fact that $\delta \circ k_m + k_m \circ \delta = \mathrm{id}$ when $\ell \geqslant 1$, and $(k_m \circ \delta)(f) = f - f(m)$
when $\ell =0$. Since $\int _M \chi_m ~d\mathrm{vol}(m) =1$ by Lemma \ref{lemma-chi}, and since $d''f = \delta f$, 
the operator $K$ satisfies the excepted homotopy relations.
\end{proof}

\begin{proof}[Proof of Proposition \ref{appendix-proposition3}] 
We keep the notations of Lemma \ref{appendix-lemma2}. Recall that the subspace $\Ker d'' \restr _{C^{k, 0}}$ is described
in the proof on Proposition \ref{appendix-proposition1}. Define a retraction 
$s: C^{*,\ell} \to (\Ker d'' \restr _{C^{k, 0}} \to 0 \to 0 \to \dots)$, by letting 
$s(f) = \int _M \chi _m \cdot f(m)~ d\mathrm{vol}(m)$ on $C^{k, 0}$, and $s = 0$ on $C^{k, \ell}$ 
when $k \geqslant 1$. Clearly it commutes with $d''$. Since $\int _M \chi_m ~d\mathrm{vol}(m) =1$, one has $s \circ i = \mathrm{id}$. 
Finally $K$ is a homotopy between $i \circ s$ and $\mathrm{id}$. 
\end{proof}

\subsection{Proof of Theorem \ref{appendix-theorem}}
Let $(D^*, d_D)$ be the complex $D^m = \oplus _{k + \ell = m} C^{k, \ell}$ with $d_D = d' + d''$.
From Propositions \ref{appendix-proposition1}, \ref{appendix-proposition2} and Lemma \ref{appendix-lemma0}, it retracts by deformation
onto $AS^{p, *}(M)$. From Propositions \ref{appendix-proposition1}, \ref{appendix-proposition3} and Lemma \ref{appendix-lemma0}, 
it retracts by deformation onto $\Omega^{p, *}(M)$. Therefore $AS^{p, *}(M)$ and $\Omega^{p, *}(M)$ are homotopy equivalent.
Their cohomologies are topologically isomorphic.
\hfill$\square$

\def\cprime{$'$}

\bigskip

\bigskip

\noindent Laboratoire Paul Painlev\'e, UMR 8524 CNRS / Universit\'e de Lille, 
Cit\'e Scientifique, B\^at. M2, 59655 Villeneuve d'Ascq, France. \\
E-mail: {\tt marc.bourdon@univ-lille.fr}.

\noindent Unit\'e de Math\'ematiques Pures et Appliqu\'ees, UMR 5669 CNRS / \'Ecole normale sup\'erieure de Lyon, 
46 all\'ee d'Italie, 69364 Lyon cedex 07,  France\\
E-mail: {\tt bertrand.remy@ens-lyon.fr}.

\end{document}